\documentclass{amsart}
\usepackage{times,mathabx,amsfonts,amssymb,amsrefs,mathrsfs,tikz}
\usetikzlibrary{decorations,decorations.pathmorphing}

\usepackage{lineno}\linenumbers
\definecolor{gray}{gray}{0.5}

\newcommand\Z{\mathbb{Z}}
\newcommand\N{\mathbb{N}}

\newcommand\Q{\mathbb{Q}}

\newcommand\conv{\stackrel{\scriptscriptstyle<}{\rightsquigarrow}}
\newcommand\vnoc{\stackrel{\scriptscriptstyle>}{\leftsquigarrow}}

\newcommand\free{{\mathbb F}}
\newcommand\relfree{{\mathbb V}}
\newcommand\cay[2]{{\mathscr C(#1,#2)}}
\newcommand\grig{G_{012}}
\newcommand\subgp{\subseteq}
\newcommand\dist{{\operatorname{dist}_{\conv}}}
\newcommand\Torsion{{\operatorname{Torsion}}}
\newcommand\variety{{\mathcal V}}
\newcommand\markedgroups{{\mathscr G}}
\newcommand\extsmallergroups{prelimit of an overgroup}
\newcommand\preforms{preforms}
\newcommand\preform{preform}
\newcommand\preformed{preformed}

\newtheorem{lemma}{Lemma}[section]

\newtheorem{proposition}[lemma]{Proposition}
\newtheorem{theorem}[lemma]{Theorem}
\newtheorem{corollary}[lemma]{Corollary}
\newtheorem{example}[lemma]{Example}

\newtheorem{maintheorem}{Theorem}

\newtheorem{mainproposition}[maintheorem]{Proposition}
\theoremstyle{definition}
\newtheorem{remark}[lemma]{Remark}
\newtheorem{definition}[lemma]{Definition}
\newtheorem{question}[lemma]{Question}

\font\manfnt=manfnt

\newcommand\twoheaddownarrow{\hbox to 0pt{\raisebox{0.3ex}{$\downarrow$}}
  \hbox to 0pt{\raisebox{-0.2ex}{$\downarrow$}}\phantom\downarrow}

\begin{document}
\title{Ordering the space of finitely generated groups}

\author{Laurent Bartholdi}
\address{L.B.: Mathematisches Institut, Georg-August Universit\"at, G\"ottingen, Germany}

\author{Anna Erschler}
\address{A.E.: C.N.R.S., D\'epartement de Math\'ematiques, Universit\'e Paris Sud, Orsay, France}

\date{January 20, 2013}

\thanks{The work is supported by the ERC starting grant 257110
  ``RaWG'', the ANR ``DiscGroup: facettes des groupes discrets'' and
  the Courant Research Centre ``Higher Order Structures'' of the
  University of G\"ottingen}

\begin{abstract}
  We consider the oriented graph whose vertices are isomorphism
  classes of finitely generated groups, with an edge from $G$ to $H$
  if, for some generating set $T$ in $H$ and some sequence of
  generating sets $S_i$ in $G$, the marked balls of radius $i$ in
  $(G,S_i)$ and $(H,T)$ coincide.

  Given a nilpotent group $G$, we characterize its connected component
  in this graph: if that connected component contains at least one
  torsion-free group, then it consists of those groups which generate
  the same variety of groups as $G$.

  The arrows in the graph define a preorder on the set of isomorphism
  classes of finitely generated groups. We show that a partial order
  can be imbedded in this preorder if and only if it is realizable by
  subsets of a countable set under inclusion.

  We show that every countable group imbeds in a group of non-uniform
  exponential growth.  In particular, there exist groups of
  non-uniform exponential growth that are not residually of
  subexponential growth and do not admit a uniform imbedding into
  Hilbert space.
\end{abstract}
\maketitle



\section{Introduction}


Our aim, in this paper, is to relate the following preorder on the set
of isomorphism classes of finitely generated groups with asymptotic
and algebraic properties of groups.

\begin{definition}\label{def:precede}
  Let $G,H$ be finitely generated groups.  We write $G\conv H$, and
  say that $G$ \emph{\preforms} $H$, if the following holds.  There
  exist a finite generating set $T$ of $H$ and a sequence of finite
  generating sets $S_1,S_2,\dots$ of $G$, with bijections $S_n\to T$
  such that, for all $R\in\N$, if $n$ is large enough then the balls
  of radius $R$ in the marked Cayley graphs of $(G,S_n)$ and $(H,T)$
  are isomorphic.

  We denote by $\cay GS$ the Cayley graph of the group $G$ with
  respect to the generating set $S$. Its edges are marked with the
  generator they correspond to.

  If $G$ \preforms\ $H$, then we also say that $H$ \emph{is
    \preformed\ by} $G$.
\end{definition}

Definition~\ref{def:precede} can be interpreted in terms of the
\emph{Chabauty-Grigorchuk topology}, also called the \emph{Cayley
  topology}, defined as follows. The \emph{space of marked groups} is
the set $\markedgroups$ of pairs $(G,S)$ with $G$ a finitely generated
group and $S$ a finite ordered generating set, considered up to group
isomorphism preserving the generating set. This is equipped with a
natural topology, two marked groups $(G,S)$ and $(G',S')$ being close
to each other if marked balls of large radius in the Cayley graphs
$\cay GS$ and $\cay G{S'}$ are isomorphic.

Chabauty considered this topological space
in~\cite{chabauty:limites}*{\S3}; he used it to describe the space of
lattices in locally compact
groups. Gromov~\cite{gromov:nilpotent}*{pages 71--72} used it to
derive an effective version of his theorem on groups of polynomial
growth. Grigorchuk ~\cite{grigorchuk:gdegree} was the first to study
this topology systematically; in particular, he used it to construct
groups of wildly-oscillating intermediate growth, by approximating
them in $\markedgroups$ by solvable groups. For generalities on the
the space of marked groups,
see~\cite{champetier-guirardel:limitgroups}.

Definition~\ref{def:precede} may then be formulated as follows:
$G\conv H$ if and only if the closure of the isomorphism class of $G$
in the Chabauty-Grigorchuk topology contains $H$.

It is essential for our definition that we consider limits in the
space of marked groups of a fixed group, letting only its generating
set vary.  Various authors have already considered limits in the space
of marked groups, not necessarily restricting to limits within one
isomorphism class.  Limits of one fixed group have been studied when
this group is free: they coincide with limits groups, as shown by
Champetier and
Guirardel~\cite{champetier-guirardel:limitgroups}*{Theorem~1.1};
see~\S\ref{ss:limitgroups} for more
references. Zarzycki~\cite{zarzycki:limitsthompson} considers groups
that are \preformed\ by Thompson's group $F$, and gives some necessary
conditions for HNN extensions to appear in this manner;
Guyot~\cites{guyot:dihedral,guyot:metabelian} considers groups that
are \preformed\ by $G$ for some metabelian groups $G$, and identifies
their closure in $\markedgroups$. On the other hand, groups that
\preform\ free groups are groups that have infinite girth for
generating sets of fixed cardinality. Olshansky and Sapir characterize
them in~\cite{olshanskii-sapir:fklike} as groups without
almost-identities, see also~\S\ref{ss:smallerfree}.

We recall that a \emph{preorder} is a binary relation $\precsim$ such
that $A\precsim C$ whenever $A\precsim B$ and $B \precsim C$ and such
that $A\precsim A$ for all $A$. If furthermore `$A\precsim B$ and
$B\precsim A$' imply $A=B$, then it is an \emph{order}. A preorder is
\emph{directed} if every finite subset has an upper bound. It is easy
to see that the relation `$\conv$' is a preorder, and that $G\conv H$
does not depend on the choice of a finite generating set in $H$ (see
Lemmas ~\ref{lem:preorder} and~\ref{choiceofgen} in the next section).
It is also not difficult to see that the restriction of this relation
to some classes of groups is an order; this happens, for example, for
residually finite finitely presented groups, such as polycyclic groups
(see Corollary~\ref{cor:polyorder}). For some other classes of groups
this is not true: for example, there exist solvable groups $G$
admitting a continuum of non-isomorphic solvable groups which are
equivalent to $G$ under our preorder, that is, which both \preform\
and are \preformed\ by $G$. Nekrashevych gave
in~\cite{nekrashevych:nue} examples of groups acting on rooted trees
which are equivalent under our preorder.

In many cases, if $A$ \preforms\ $B$, then $A$ ``looks smaller''
than $B$. Simple examples of this kind include: $\Z^m \conv \Z^n$ if
and only if $m\le n$; free groups satisfy $\free_m \conv\free_n$ if
and only if $m\le n$; and the $n$-generated free groups $\relfree_n$
in the variety generated by a torsion-free nilpotent group of
nilpotency class $c$ satisfy, for $m, n \ge c$, the same relation
$\relfree_m \conv\relfree_n$ if and only $m \le n$, see
Theorem~\ref{thm:A}.  On the other hand, it may happen for $A$
that \preform\ $B$ that the growth of $A$ is larger than the growth of
$B$; we consider this in more detail in~\S\ref{ss:introgrowth}.

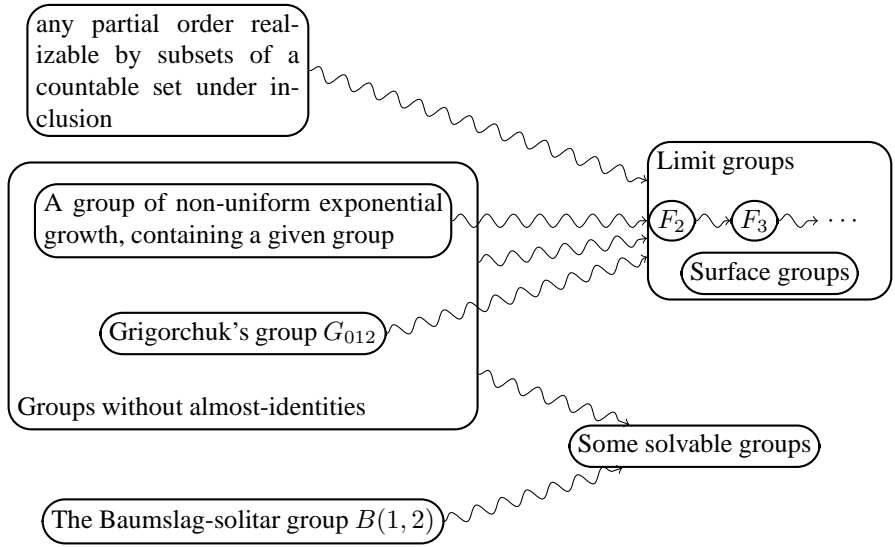
\begin{figure}[t]
\begin{tikzpicture}[box/.style={rectangle,rounded corners=3mm,thick,draw},
  snake/.style={->,decorate,decoration=snake}]
  \node[box] (order) at (0,10) {\parbox{35mm}{any partial order realizable by subsets of a countable set under inclusion}};
  \node[box] (f2) at (6.7,8) {$F_2$};
  \node[box] (f3) at (7.8,8) {$F_3$};
  \node (fn) at (9,8) {$\cdots$};
  \node[box] at (8,7.3) {Surface groups};
  \node[box] (lg) at (8,8) {\parbox{3cm}{Limit groups\\[12mm]}};
  \node[box] (gg) at (1,6.5) {Grigorchuk's group $\grig$};
  \node[box] (qi) at (1,7) {\parbox{6cm}{\vspace{3cm}Groups without almost-identities}};
  \node[box] (nue) at (1,8) {\parbox{5.3cm}{A group of non-uniform exponential growth, containing a given group}};
  \node[box] (sol) at (7,5) {Some solvable groups};
  \node[box] (bs) at (1,4) {The Baumslag-solitar group $B(1,2)$};
  \path (f2) edge[snake] (f3)
  (f3) edge[snake] (fn)
  (order.east) edge[snake] (lg)
  (gg.east) edge[snake] (lg)
  (nue) edge[snake] (lg)
  (qi) edge[snake] (lg)
  (qi) edge[snake] (sol)
  (bs.east) edge[snake] (sol);
\end{tikzpicture}
\caption{Some classes of groups and their relationship under $\conv$}
\end{figure}

\subsection{The structure of components}
We view $\conv$ as specifying the edge set of an oriented graph with
vertex set the isomorphism classes of finitely generated groups. In
studying this graph, we may consider independently the \emph{connected
  components} of its underlying unoriented graph. What do they look
like? Which components admit an initial vertex?  a terminal vertex?
Given a connected component, does it have an upper bound? What is the
group of preorder preserving bijections of a given component? Which
groups' \emph{strongly connected component} are reduced to points, or
have the cardinality of the continuum?

Unlike some other natural preorders, such as ``being a subgroup'',
``being a quotient group'', or ``being larger'' in the sence of Pride
($G\succsim_p H$ if $H_1$ is a quotient of $G_1$, for respective
quotients $G_1,H_1$ of finite-index subgroups of $G,H$ by finite
normal subgroups, see \cites{pride:large,thomas:largeness}), the
preorder that we consider in this paper has infinitely many connected
components.  An easy example is the connected component of $\Z$: it
contains all infinite abelian groups, and we describe the group of the
order preserving bijections of this component in
Proposition~\ref{prop:aut_ab}.

For a nilpotent torsion-free $G$ group, its connected component is
closely related to groups that generated the same variety as $G$.

\begin{maintheorem}[= Proposition~\ref{mainpropositionnilpotent}]\label{thm:A}
  Let $G$ be a finitely generated nilpotent group such that $G$ and
  $G/\Torsion(G)$ generate the same variety (i.e.\ satisfy the same
  identities). Then, for all $k\in\N$ large enough, $G$ \preforms\
  the relatively free group of rank $k$ in that variety.
\end{maintheorem}
In particular, every finite set of such nilpotent groups has a
supremum with respect to our preorder. We believe, in fact, that this
last statement holds for all virtually nilpotent groups.  However, if
a nilpotent group $G$ is not torsion-free, the connected component of
$G$ can be much smaller that the set of (isomorphism classes) of
groups generating the same varitety as $G$, see
Corollary~\ref{cor:nilpotentcomponents}.

We show, on the other hand, that the preorder types that can occur are
quite general, even within solvable groups of class $3$, or within
groups that \preform\ free groups:

\begin{maintheorem}[= Corollary~\ref{cor:order} and Remark~\ref{rem:order}]\label{thm:B}
  Let $(X,\precsim)$ be a preorder. Then
  $(\markedgroups/{\cong},\conv)$ contains $(X,\precsim)$ as a
  subpreorder if and only if $X$ has cardinality at most the
  continuum, and all the partial orders it contains are imbeddable in
  the partial order of subsets of $\mathscr B$ under inclusion, for a
  countable set $\mathscr B$.
\end{maintheorem}

Furthermore, the imbedding of $(X,\precsim)$ can then be chosen to be
within the set of isomorphism classes of solvable groups of solubility
class $3$  or, alternatively, within the set
of isomorphism classes of groups that \preform\ $\free_3$.

Thomas studies in~\cite{thomas:largeness} the complexity, with respect
to the Borelian structure on $\markedgroups$, of Pride's ``largeness''
preorder and of the ``being a quotient'' preorder. He shows that these
preorders are high in the Borel hierarchy (namely, \emph{$\mathbf
  K_\sigma$-universal}). The preorder $\conv$ differs from the above
mentioned preorders even if we forget the underlying Borelian
structure: the quotients and largeness preorders have chains with
cardinality the continuum, while (by Theorem~\ref{thm:B}) chains for
$\conv$ are countable.




\subsection{Groups larger or smaller than a given group}
Given a group $G$, how many groups \preform\ $G$?  How
many groups are \preformed\ by $G$? How big is the connected
component of $G$? What is its diameter? 

We note that, if a group $G$ is virtually nilpotent, then its component is
countable. The number of groups that are \preformed\ by $G$ is 
countably infinite.

If $G$ is a free group, a surface group, or more generally a
non-abelian limit group (see~\S\ref{ss:limitgroups}), then there are
countably many groups that are \preformed\ by $G$,
see~\cites{sela:diophantine1,kharlampovich-myasnikov:iavofg1}.
However, the connected component of $G$ has the cardinality of the
continuum, see Example~\ref{ex:continuum}.

We study the groups that \preform\ free groups. Schleimer
considered groups of unbounded girth (there are generating sets such
that the smaller cycle in the Cayley graph is arbitrarily long) in an
unpublished note~\cite{schleimer:girth}, and they are intimately
connected to groups that \preform\ free groups, see
Question~\ref{qu:girth}. The latter are groups that do not satisfy an
\emph{almost-identity}~\cite{olshanskii-sapir:fklike}: a word whose
evaluation vanishes on every generating set.  Olshanskii and Sapir
show in ~\cite{olshanskii-sapir:fklike} that there are groups with
non-trivial quasi-identities among groups satisfying no non-trivial
identitity.

In~\S\ref{ss:abert}, we modify a criterion by
Ab\'ert~\cite{abert:nonfree} about groups without identities to
determine when a group has no almost-identity. This lets us answer
negatively a question by
Schleimer~\cite{schleimer:girth}*{Conjecture~6.2} that groups of
unbounded girth have exponential word growth
(see~\S\ref{ss:introgrowth} for the definition of growth):
\begin{maintheorem}[= Corollary~\ref{cor:grig<free}]
  The first Grigorchuk group $\grig$ \preforms\ $\free_3$.
\end{maintheorem}

Extending an argument by Akhmedov (see~\cite{akhmedov:girth}), we give
a criterion for a wreath product with infinite acting group to \preform\
a free group:
\begin{mainproposition}[= Proposition~\ref{prop:wreath}]
  Let $G$ and $H$ be finitely generated groups, and suppose that $H$ is
  infinite. Then the restricted wreath product $G\wr H:=G^{(H)}\rtimes
  H$ \preforms\ a free group if and only if at least one of the
  following conditions holds:
  \begin{enumerate}
  \item $G$ does not satisfy any identity;
  \item $H$ does not satisfy any almost-identity.
  \end{enumerate}
\end{mainproposition}
From this, we deduce (see Remark~\ref{re:atleastthree}) that the
connected component of the free group has diameter at least $3$; this
is in contrast with the nilpotent case, see Theorem~\ref{thm:A}. There
are solvable groups, and infinite free Burnside groups, at distance
$2$ from a free group.

See also subsection~\ref{limitsandprelimits} where we discuss groups
that \preform\ a group containing a given subgroup.

\subsection{Growth of groups}\label{ss:introgrowth}
We finally give in~\S\ref{ss:growth} some new examples of groups of
non-uniform exponential growth. Recall that, for a group $G$ generated
by a set $S$, its \emph{growth function} counts the number
$\nu_{G,S}(R)$ of group elements expressible as a product of at most
$R$ generators. The group has \emph{exponential growth} if
$\lambda_{G,S}:=\lim\sqrt[R]{\nu_{G,S}(R)}>1$ and \emph{subexponential
  growth} otherwise; it then has \emph{polynomial growth} if
$\nu_{G,S}$ is dominated by a polynomial, and \emph{intermediate
  growth} otherwise. The existence of groups of intermediate growth
was asked by Milnor in~\cite{milnor:5603}, and answered by Grigorchuk
in~\cite{grigorchuk:gdegree}, by means of his group $\grig$.

If $G$ has exponential growth, then it has \emph{uniform exponential
  growth} if furthermore $\inf_S\lambda_{G,S}>1$. The existence of
groups of non-uniform exponential growth was asked by Gromov
in~\cite{gromov:metriques}*{Remarque~5.12}; see
also~\cite{harpe:uniform}. The first examples were constructed by
Wilson~\cite{wilson:ueg}; see
also~\cites{bartholdi:nueg,nekrashevych:nue,wilson:fnueg}.

\begin{maintheorem}[= Corollary~\ref{cor:imbednueg}]\label{thm:E}
  Every countable group may be imbedded in a group $G$ of non-uniform
  exponential growth.

  Furthermore, let $\alpha\approx0.7674$ be the positive root of
  $2^{3-3/\alpha} + 2^{2-2/\alpha} + 2^{1-1/\alpha} = 2$.  Then $G$
  may be required to have the following property: there is a constant
  $K$ such that, for any $R>0$, there exists a generating set $S$ of
  $G$ with
  \[\nu_{G,S}(r) \le \exp(K r^\alpha)\text{ for all }r\le R.
  \]
\end{maintheorem}

Theorem~\ref{thm:E} implies the existence of groups of non-uniform
exponential growth that do not imbed uniformly into Hilbert space;
this answers a question by Brieussel~\cite{brieussel:entropy}*{after
  Proposition 2.5}, who asked whether there exist groups of
non-uniform exponential growth without the Haagerup property.  We also
construct groups of non-uniform exponential growth that admit
infinitely many distinct intermediate growth functions at different
scales.  Moreover, these examples can be constructed among groups that
\preform\ free groups and groups of intermediate growth.

The idea of the proof of Theorem~\ref{thm:E} is as follows.  We denote
by $\grig$ the first Grigorchuk group. It acts on the infinite binary
tree $\{0,1\}^*$ and its boundary $\{0,1\}^\infty$. We denote by $X$
the orbit $\grig\cdot1^\infty$.  We prove in Corollary~\ref{cor:nueg}
that the group $G\wr_X\grig$ has non-uniform exponential growth
whenever $G$ is a group of exponential growth.  To prove
Corollary~\ref{cor:nueg} we show that $G\wr_X\grig$ \preforms\ a
group of intermediate growth. (In fact, all known examples of groups
of non-uniform exponential growth \preform\ groups of
intermediate growth, though the corresponding group of intermediate
growth is not always given explicitly by their construction ; for more
on this see Question~\ref{qu:nue=>seg}).

\subsection{Acknowledgments}
The authors are grateful to Yves de Cornulier, Slava Grigorchuk,
Fr\'ed\'eric Paulin and Pierre de la Harpe for their comments on an
earlier version of this manuscript; to Olga Kharlampovich for having
corrected an inaccuracy in our understanding of limit groups; to Misha
Gavrilovich for enlightening discussions; and to Simon Thomas and
Todor Tsankov for their generous explanations on Borel relations.


\section{First properties and examples}
\begin{lemma}[A special case of~\cite{champetier-guirardel:limitgroups}*{Proposition~2.20}] \label{choiceofgen}
  The ``for some generating set $T$'' in Definition~\ref{def:precede}
  may be changed to ``for every generating set $T$''.
\end{lemma}
\begin{proof}
  Assume $G\conv H$, that $T$ generates $H$ and that $\cay G{S_n}$
  coincides with $\cay HT$ on ever larger balls. Write $\tau_n:T\to
  S_n$ the bijections.

  Let $T'$ be another generating set of $H$; write every $t\in T'$ as
  a word $w_t$ over $T$. Let $k$ be the maximum of the lengths of the
  $w_t$. Consider the generating sets $S'_n=\{w_t(\tau_n)\colon t\in
  T\}$ of $G$ obtained by replacing each $T$-letter in $w_t$ by its
  corresponding element $\tau_n(t)\in G$.

  Then, if $\cay G{S_n}\cap B(1,R)$ is isomorphic to $\cay HT\cap
  B(1,R)$, then $\cay G{S'_n}\cap B(1,\lfloor R/k\rfloor)$ is
  isomorphic to $\cay H{T'}\cap B(1,\lfloor R/k\rfloor)$, since they
  are respective subsets in the isomorphic graphs $\cay G{S_n}\cap
  B(1,R)$ and $\cay HT\cap B(1,R)$.
\end{proof}

\begin{lemma}\label{lem:preorder}
  The relation $\conv$ is a preorder.
\end{lemma}
\begin{proof}
  It is clear that $G\conv G$ holds for all groups $G$.
  
  Consider now $G\conv H\conv K$, and let $U$ be a generating set for
  $K$. There are then generating sets $T_n$ for $H$, in bijection with
  $U$, such that $\cay H{T_n}$ and $\cay KU$ agree in ever larger
  balls. For each $n$, there are generating sets $S_{mn}$ for $G$, in
  bijection with $T_n$, such that $\cay G{S_{mn}}$ and $\cay H{T_n}$
  agree in ever larger balls.

  Therefore, the generating sets $S_{nn}$, which are in bijection with
  $U$, are such that $\cay G{S_{nn}}$ and $\cay KU$ agree in ever
  larger balls, which shows $G\conv K$.
\end{proof}

Let $\free$ be the free group on infinitely many generators
$x_1,x_2,\dots$, and consider the space $\markedgroups$ of finitely
generated groups $(G,T)$ with marked generating set. This marking may
be given by a homomorphism $\free\twoheadrightarrow G$ such that almost
all $x_n$ map to $1$; and this identifies $\markedgroups$ with the set of
normal subgroups of $\free$ containing almost all the $x_n$. This turns
$\markedgroups$ into a locally compact Polish space. In this alternative
terminology, we have the obvious
\begin{lemma}
  Let $G,H$ be finitely generated groups. Then $G\conv H$ if and only
  if for some (hence all) generating set $T$, the marked group
  $(H,T)$ belongs to the closure of $\{(G,S)\colon S\text{ generates
  }G\}$ in $\markedgroups$.
\end{lemma}

We observe that if $G\conv H$ and either $G$ or $H$ are finite, then
$G=H$. We thus restrict ourselves to infinite, finitely generated
groups.

\begin{lemma}\label{lem:fp}
  Let $G$ be a finitely generated group, and let $H$ be a finitely
  presented group.  If $G\conv H$, then $G$ is a quotient of $H$.
\end{lemma}
\begin{proof}
  Let $T$ be a generating set of $H$, and let $R$ be the maximal
  length of $H$'s relators in that generating set. If $G\conv H$, then
  there exists a generating set $S$ for $G$ such that $\cay GS$ and
  $\cay HT$ coincide in a ball of radius $R$; so all relations of $H$
  hold in $T$.
\end{proof}

We note (\cite{champetier-guirardel:limitgroups}*{Example~2.4(e)})
that every residually finite group is a limit of finite groups;
however, the closure of the set of finite groups in $\markedgroups$
has not been convincingly identified.

It has been shown by Shalom~\cite{shalom:rigidity} that every group
$G$ with Kazhdan's property (T) is a quotient of a finitely presented
group with Kazhdan's property (T). Therefore,
\begin{lemma}[\cite{champetier-guirardel:limitgroups}*{Proposition~2.15}]
  If $G\conv H$ and $G$ does \emph{not} have Kazhdan's property (T),
  then neither does $H$.\qed
\end{lemma}

There are isolated points in the space of groups; they are studied
in~\cite{cornulier-guyot-pitsch:isolated}.  Clearly, isolated groups
are minimal elements for $\conv$; but the converse is not true. For
example, $\Z$ and $\Z\oplus\Z/p\Z$ are minimal, but none of them is
isolated.

\subsection{Partial orders} On some classes of groups, the relation
$\conv$ is also antisymmetric, and therefore defines a partial
order. Recall that a group $G$ is Hopfian if every epimorphism
$G\twoheadrightarrow G$ is an automorphism.
\begin{lemma}\label{lem:fpresfinite}
  Among Hopfian, finitely presented groups, $\conv$ is an order
  relation. More generally, if $G$ and $H$ are finitely presented
  groups with $G\conv H\conv G$ and $G$ is Hopfian, then $G$ and $H$
  are isomorphic.
\end{lemma}
\begin{proof}
  From $G\conv H$ and Lemma~\ref{lem:fp} we deduce that $G$ is a
  quotient of $H$; and similarly $H$ is a quotient of $G$. Therefore
  we have epimorphisms $G\twoheadrightarrow H\twoheadrightarrow G$,
  and since $G$ is Hopfian these epimorphisms are isomorphisms.
\end{proof}

\begin{corollary}\label{cor:polyorder}
  The relation $\conv$ is an order relation on polycyclic groups, and
  on limit groups.
\end{corollary}
\begin{proof}
  Polycyclic groups are known to be finitely presented and residually
  finite.  We will recall some known facts about limit groups
  in~\S\ref{ss:limitgroups}; for the proof of the corollary it
  suffices to know that limit groups are residually free and therefore
  residually finite; and that they are finitely presented.

  Since residually finite groups are Hopfian
  (see~\cite{malcev:matrix}), the corollary follows from
  Lemma~\ref{lem:fpresfinite}.
\end{proof}

\subsection{Identities and universal statements}\label{ss:identities}
Let $G$ be a group. An \emph{identity for $G$} is a non-trivial word
$w(x_1,x_2,\dots)$ in the free group on countably many generators,
such that $w(g_1,g_2,\dots)=1$ for every choice of $g_i\in G$. Note
that $w$ is really a word in finitely many of the $x_i$'s, namely
$w=w(x_1,\dots,x_n)$ for some $n\in\N$.

An identity for $G$ is really the following universal sentence.
`$\forall g_1,g_2\dots(w=1)$'. More generally, any well-formed
expression made of conjunctions, disjunctions, equalities, and universal
quantifiers, is a \emph{positive universal sentence}. If furthermore
negations are allowed, it is a \emph{universal sentence}. The
\emph{variety} generated by a group $G$ is the set of identities that
it satisfies; and its \emph{(positive) universal theory} is the set of
(positive) universal sentences that it satisfies.

For example, consider the group $G=\langle x,y,z\mid [x,y]z^{-1}, z^2,
[x,z], [y,z]\rangle$. It satisfies the identity $[x_1,x_2]^2$. It also
satisfies the positive universal statement
\[\forall x_1,\dots,x_4([x_1,x_2]=1\vee[x_1,x_3]=1\vee\cdots\vee[x_3,x_4]=1).
\]
As a last example, limits groups are known to be
``commutative-transitive''; this is the universal statement
\begin{equation}\label{eq:commtrans}
  \forall x,y,z([x,y]=1\wedge[y,z]=1\Rightarrow[x,z]=1).
\end{equation}
Note that this statement is not positive; rewriting it in terms of the
primitives $\vee,\wedge,\neg$ gives $\forall
x,y,z(\neg([x,y]=1\wedge[y,z]=1)\vee[x,z]=1)$. An example of a
positive statement appears in Example~\ref{ex:n22xn22}. For more
details relating logic to the space of marked groups,
see~\S\ref{ss:limitgroups}
and~\cite{champetier-guirardel:limitgroups}*{\S5}. In particular, the
first assertion of the following lemma
is~\cite{champetier-guirardel:limitgroups}*{Proposition~5.2}.

\begin{lemma}\label{lem:identity}
  \begin{enumerate}
  \item If $G\conv H$ and $G$ satisfies a universal statement
    (e.g., an identity), then $H$ satisfies it too;
  \item If $G\conv H$ and $H$ is a finitely presented group satisfying
    a positive universal statement, then $G$ satisfies it too;
  \item If $G\conv H$ and $G$ is torsion-free, then $H$ is torsion
    free. More generally, if $F$ is a finite subgroup of $H$, then $F$
    imbeds in $G$.
  \end{enumerate}
\end{lemma}

\begin{remark}
  It is essential not to allow negations in~(2): a group with torsion, and moreover a torsion group, 
  can \preform\ a finitely presented torsion-free group --- e.g.,
  Grigorchuk's group $G$ \preforms\ $\free_3$.  In fact, if $G
  \conv\free_n$ for some $n$, then $G$ has the same positive universal
  theory as $\free$. However, $G$ is universally equivalent to $H$ if
  and only if $G$ is a non-abelian limit group of Sela
  (see~\S\ref{ss:limitgroups}, that is, if $\free\conv G$.

  The lemma implies in particular that if $G$ is virtually nilpotent,
  then every group in the same connected component has the same
  language of positive universal statements. However, in any such
  connected component there are groups that are not universally
  equivalent to $G$.
\end{remark}

\begin{proof}
  Ad (1): consider a universal statement satisfied in $G$; it is of
  the form $\forall x_1,\dots,x_n(E)$ for a boolean expression $E$
  made of identities $w_1,\dots,w_\ell$. Let $R$ be the maximal length
  $w_1,\dots,w_\ell$.

  Consider arbitrary $h_1,\dots,h_n\in H$. Extend $\{h_1,\dots,h_n\}$
  to a generating set $T$ of $H$, and find a generating set $S$ of $G$
  such that the balls of radius $R$ in $\cay GS$ and $\cay HT$
  coincide. Let $g_1,\dots,g_n$ be the generators of $G$ that
  correspond to $h_1,\dots,h_n$ respectively. Then $w_i$ traces a path
  in $\cay GS$ that remains in an $R$-neighbourhood of the origin, so
  $w_i$ traces a closed loop in $\cay GS$ if and only if it traces a
  closed loop in $\cay HT$; therefore,
  $w_i(h_1,\dots,h_n)=1\Leftrightarrow w_i(g_1,\dots,g_n)=1$, so
  $E(h_1,\dots,h_n)$ follows from $E(g_1,\dots,g_n)$.

  Ad (2): Lemma~\ref{lem:fp} shows that $G$ is a quotient of $H$; and
  positive universal statements are preserved by taking quotients.

  Ad (3): consider a finite group $F$. Then the fact that $F$ is
  \emph{not} a subgroup of $G$ is a universal statement: writing
  $f_1,\dots,f_k$ the elements of $F$, with multiplication table
  $f_if_j=f_{m(i,j)}$, the statement is $\forall
  g_1,\dots,g_k(g_i=g_j\text{ for some }i\neq j\vee g_ig_j\neq
  g_{m(i,j)}\text{ for some }i,j)$. Therefore (3) follows from~(1).
\end{proof}

\subsection{Varieties} We defined varieties in~\S\ref{ss:identities}
as collections of identities. Alternatively
(see~\cite{neumann:varieties}), it is a family of groups closed under
taking subgroups, quotients and cartesian products, namely the class
$\variety$ of all the groups that satisfy these identities. The
variety $\variety$ is \emph{finitely based} if it may be defined by
finitely many identities. It is \emph{finite} if all finitely
generated groups in the variety are finite. For a group $G$, one
defines $\variety(G)=\langle w_i(g_1,g_2,\dots)\colon
i\ge1,\,g_1,g_2,\dots\in G\rangle$, the \emph{verbal subgroup of $G$}
corresponding to $\variety$; thus $\variety(G)=1$ if and only if $G$
belongs to the variety. The $k$-generated \emph{relatively free group}
is $\relfree_k:=\free_k/\variety(\free_k)$; it belongs to $\variety$,
and every $k$-generated group is $\variety$ is a quotient of
$\relfree_k$. A direct consequence of Lemma~\ref{lem:identity}(1) is
the
\begin{lemma}
  If $G\conv H$ and $G$ belongs to $\variety$, then $H$ belongs to
  $\variety$.\qed
\end{lemma}

We will consider, in later sections, the restriction of the relation
$\conv$ to groups belonging to a variety. Just as $\markedgroups$ is a
topology on the normal subgroups of $\free_k$, there is a topology
$\markedgroups(\variety)$ on the normal subgroups of $\relfree_k$, or
equivalently on the normal subgroups of $\free_k$ that contain
$\variety(\free_k)$. Directly from the definitions,
\begin{lemma}[\cite{champetier-guirardel:limitgroups}*{Lemma 2.2}]
  The natural map
  $\markedgroups(\variety)\hookrightarrow\markedgroups$ is a
  homeomorphism on its image, and that the image is closed if and only
  if $\relfree_k$ is finitely presented for all $k\in\N$.\qed
\end{lemma}

\begin{lemma}\label{lem:verbalsg}
  Let $\variety$ be a finite variety. If $G\conv H$, then
  $\variety(G)\conv\variety(H)$.
\end{lemma}
\begin{proof}
  Let $H$ be generated by a set $T=\{h_1,\dots,h_k\}$ of cardinality
  $k$, and let $\free_k$ denote the free group on $k$ generators
  $x_1,\dots,x_k$. Then $\variety(\free_k)$ admits a generating set
  of the form $w(v_1,\dots)$ for some identities $w$ in $\variety$
  and some $v_1,\dots\in\free_k$. Then $\variety(H)$ is generated by
  the set $T'$ of all corresponding $w(v_1(h_1,\dots,h_k),\dots)$.

  Consider a generating set $S=\{g_1,\dots,g_k\}$ of $G$, such that
  $\cay GS$ coincides with $\cay HT$ in a large ball; then
  $S'=\{w(v_1(g_1,\dots,g_k),\dots),\dots\}$ generates $\variety(G)$,
  and the Cayley graphs $\cay{\variety(G)}{S'}$ coincides with
  $\cay{\variety(H)}{T'}$ in a large ball.
\end{proof}

Given a variety $\variety$, the \emph{verbal product} of groups
$G_1,G_2,\dots,G_n$ is defined as follows: first set
$G=G_1*G_2\cdots*G_n$ the free product; then
\[\prod_{\variety}G_i=\frac{G}{\variety(G)\cap\langle[g_i,g_j]\colon
  g_i\in G_i^G,g_j\in G_j^G,i\neq j\rangle}.
\]
For example, if $\variety$ is the variety of all groups, then
$\prod_{\variety}$ is the free product; while if $\variety$ is the
variety of abelian groups, then $\prod_{\variety}$ is the direct
product.

Recall that the \emph{wreath product} of two groups $G_1,G_2$ is
\[G_1\wr G_2=\{f:G_2\to G_1\text{ of finite support}\}\rtimes G_2,\]
where   $G_2$  acts by shift on functions $G_2\to G_1$.

\begin{lemma}\label{lem:products}
  Let $G_1,G_2,H_1,H_2$ be groups, and assume $G_1\conv H_1$ and
  $G_2\conv H_2$. Then
  \begin{enumerate}
  \item $G_1\times G_2\conv H_1\times H_2$;
  \item $G_1*G_2\conv H_1*H_2$;
  \item Let $\variety$ be a variety of groups. Then $\prod_\variety G_i\conv\prod_{\variety}H_i$;
  \item $G_1\wr G_2\conv H_1\wr H_2$.
  \end{enumerate}
\end{lemma}
\begin{proof}
  We start by~(2), and argue that, for arbitrarily large $R$, we can
  make balls of radius $R$ agree in respective Cayley graphs. For all
  $i\in\{1,2\}$, let $T_i$ generate $H_i$, and let $S_i$ generate
  $G_i$ in such a manner that balls of radius $R$ coincide in
  $\cay{G_i}{S_i}$ and $\cay{H_i}{T_i}$. Then $T:=\bigsqcup T_i$
  generates $H:=\bigast_i H_i$, and the corresponding set
  $S:=\bigsqcup S_i$ generates $\bigast_i G_i$. Balls of radius $R$
  coincide in $\cay GS$ and $\cay HT$.

  Ad (3), the relations imposed on $\bigast_i G_i$ and $\bigast_i H_i$
  are formally defined by $\variety$, so again balls of radius $R$ in
  $\cay{\prod_{\variety}G_i}{S}$ and $\cay{\prod_\variety H_i}{T}$
  coincide.

  (1) is a special case of~(3).

  Ad (4), note that the relations giving $G_1\wr G_2$ from
  $G:=G_1*G_2$ are $[x_1^{x_2},y_1^{y_2}]$ for all $x_1,y_1\in G_1$
  and $x_2,y_2\in G_2\setminus\{1\}$. These relations do not exactly
  define a varietal product; but nevertheless there is a bijection
  between non-trivial elements of norm $\le R$ in $G_2$ and $H_2$, and
  between elements of norm $\le R$ in $G_1$ and $H_1$. The result
  again follows.
\end{proof}

Note that in~(1) we can have $G_1\times C\conv H_1\times C$ without
having $G_1\conv H_1$. We we examine more carefully this for abelian
groups in~\S\ref{ss:abelian}:
\begin{example}
  We have $1\times\Z\conv\Z\times\Z$, yet $1$ doesn't \preform\ $\Z$.

  For $A=\Z/6\times\Z$, $B=\Z/35\times\Z$, $C=\Z/10\times\Z$,
  $D=\Z/21\times\Z$, we also have $A\times B\conv C\times D$ while
  $A,B,C,D$ are mutually incomparable.
\end{example}
\begin{proof}
  Consider $\{(1,0),(0,1)\}$ a generating set of $\Z\times\Z$, and,
  for arbitrary $R\in\N$, the generating set $\{(0,1),(0,2R+1)\}$ of
  $1\times\Z$. Their Cayley graphs agree on a ball of radius $R$.

  For the second claim, note that $A\times B$ is isomorphic to
  $C\times D$, but for any two groups among $A,B,C,D$, none is a
  quotient of the other.
\end{proof}

Similarly, in~(2) we can have $G_1*C\conv H_1*C$ without having
$G_1\conv H_1$. We will examine more closely the situation of free
groups in~\S\ref{ss:limitgroups}; here and in the sequel we use the
notation $\free_k$ for free groups on $k$ generators. For now, we just
mention the
\begin{example}
  Let $G$ be a $k$-generated group. Then, for every $m\ge2$, the free
  product $G*\free_m$ \preforms\ $\free_{k+m}=\free_k*\free_m$; yet
  $G$ need not \preform\ $\free_k$, for example if $G$ satisfies an
  identity.
\end{example}
\begin{proof}
  Let $S$ generate $\free_k$, let $T$ generate $\free_m$, and let
  $\{g_1,\dots,g_k\}$ generate $G$. Then $S\sqcup T$ generates
  $\free_k*\free_m$. In $\free_m$, there exist elements
  $w_1,\dots,w_k$ such that no relation among them and $T$, of length
  $\le R$, holds; consider the generating set
  $\{g_1w_1,\dots,g_kw_k\}\sqcup T$ of $G*\free_m$. Then no relation
  of length $\le R$ holds among them.
\end{proof}

Note finally that in~(4) we may have $G_1\wr C\conv H_1\wr C$ without
having $G_1\conv H_1$; see~\S\ref{ss:wreath} for more examples:
\begin{example}\label{ex:wreath}
  Consider $A,B$ arbitrary groups, and an infinite group $C$. Then
  $(A*B)\wr C\conv(A\times B)\wr C$.

  On the other hand, if $A$ and $B$ are non-trivial, finitely
  presented, and each satisfies an identity, then $A*B$ does not
  satisfy the identities of $A\times B$, so $A*B$ doesn't \preform\
  $A\times B$ by Lemma~\ref{lem:identity}(2).
\end{example}
\begin{proof}
  Let $S,T,U$ be generating set of $A,B,C$ respectively. Then, as
  generating set of $(A\times B)\wr C$, we consider $S'\sqcup
  T'\sqcup U$, in which $S'$ corresponds to the generators of $A$
  supported at $1\in C$, and similarly for $T'$.

  For arbitrary $R\in\N$, choose $x\in C$ of norm $>R$, and consider
  the following generating set $S''\sqcup T''\sqcup U$ of $(A*B)\wr
  H$. The copy $S''$ of $S$ corresponds to the generators of $A$
  supported at $1\in C$, while the copy $T''$ corresponds to the
  generators of $T$ supported at $x$.

  Both $(A\times B)\wr C$ and $(A*B)\wr C$ are quotients of $A*B*C$;
  in both cases, all relations of the form $[s_1^h,s_2]$ and
  $[t_1^h,t_2]$ are imposed for all $h\neq1$ and $s_i\in S',t_i\in
  T'$, respectively $s_i\in S'',t_i\in T''$. However, in the former
  case, all relations of the form $[s^h,t]$ are also imposed for all
  $h\in H$ and $s\in S',t\in T'$. In the latter case, these relations
  are only imposed for $h\neq x$ and $s\in S'',t\in T''$. However,
  this distinction is invisible in the ball of radius $R$.
\end{proof}

\subsection{Limits and prelimits of groups with a given subgroup or
  quotient}\label{limitsandprelimits}
We start by the following straightforward lemma.

\begin{lemma}
  If $A\conv B$ and $A$ is a subgroup of $G$, then there exists a
  group $H$ containing $B$ as a subgroup and satisfying $G\conv H$:
  \[\begin{matrix}G & \textcolor{gray}{\conv} & \textcolor{gray}{H}\\
    \cup && \textcolor{gray}{\cup}\\
    A &\conv & B.\end{matrix}\]
\end{lemma}
\begin{proof}
  Consider finite generating sets $S_n$ of $A$ and $T$ of $B$ such
  that $(A,S_n)$ converges to $(B,T)$ in the space $\markedgroups$ of
  marked groups, as $n\to\infty$. Let $S$ be a finite generating set
  of $G$. Set $S'_n=S\sqcup S_n$; these define finite generating sets
  of $G$. Consider a subsequence $(n_k)$ such that $(G,S'_{n_k})$
  converges in $\markedgroups$; denote its limit by $(H,U\sqcup V)$.

  In particular, $(A,S_{n_k})$ converges to the subgroup $\langle
  V\rangle$ of $H$. Since $(A,S_n)$ converges to $B$, we conclude that
  $\langle V\rangle$ is isomorphic to $B$.
\end{proof}

\begin{lemma}\label{lem:quotients}
  If $A\conv B$ and $A$ is a quotient of $G$, then there exists a
  group $H$ with $G\conv H$ and $B$ is a quotient of $H$:
  \[\begin{matrix}G & \textcolor{gray}{\conv} & \textcolor{gray}{H}\\
    \twoheaddownarrow && \textcolor{gray}{\twoheaddownarrow}\\
    A &\conv & B.\end{matrix}\]
\end{lemma}
\begin{proof}
  Let $A,B$ be $k$-generated, with $T$ a generating set for $B$.
  Since $A$ \preforms\ $B$, there exists a sequence of
  generating sets $S_n$ of cardinality $k$ such that $(A,S_n) \to
  (B,T)$. Without loss of generality, we may assume $1\in S_n$ for all
  $n\in\N$.

  Let $\pi\colon G\twoheadrightarrow A$ be the given
  epimomorphism. Let $G$ be $\ell$-generated. Then for each $n\in\N$
  there exists a generating set $S'_n=S''_n\sqcup S'''_n$ of $G$ such
  that $S''_n$ maps bijectively to $S_n$ under $\pi$ and $S'''_n$ maps
  to $1\in A$ and has cardinality $\ell$. Indeed first choose a
  generating set $S'$ for $G$ of cardinality $\ell$; then, for each
  $n\in\N$, choose an arbitrary lift $S'_n$ of $S_n$; and multiply
  each $g\in S'$ by an appropriate word in $S'_n$ to obtain $S'''_n$
  mapping to $1$.

  Passing if need be to a subsequence, we can assume that $(G,S'_n)$
  converges in the space $\markedgroups$ of marked groups. Denote the
  limit of the subsequence by $(H,T')$, again with decomposition
  $T'=T''\sqcup T'''$. Let us construct an epimorphism $\rho\colon
  H\twoheadrightarrow B$, showing that $B$ is a quotient of $H$.
  Recall that $T''$ in naturally in bijection with $T$, via $S''_n$ and
  $S_n$. We define $\rho$ on $T''$ by this bijection, and put
  $\rho(t)=1$ for all $t\in T'''$.

  To prove that $\rho$ is a homomorphism, consider a word
  $w(x_1,\dots,x_{k+\ell})$ with $w(T')=1$ in $H$. Since $(G,S'_n)$
  converges to $(H,T')$, for sufficiently large $n\in\N$ we have
  $w(S'_n)=1$ in $G$. Let $v(x_1,\dots,x_k)$ denote the word obtained
  from $w$ be deleting its letters $x_{k+1},\dots,x_{k+\ell}$. Since
  $\pi$ is a homomorphism, we then have $v(S_n)=1$, and therefore in
  the limit $v(T)=1$. This is precisely the result of computing
  $\rho(w(T'))$ letter by letter.

  Finally, $T$ is in the image of $\rho$ so $\rho$ is surjective.
\end{proof}

We may improve on Lemma~\ref{lem:quotients} in case the quotient is by
a verbal subgroup:
\begin{lemma}\label{lem:verbalquotients}
  Let the group $G$ be generated by a set of cardinality $k$, and let
  $\variety$ be a variety.  If $G/\variety(G)\conv\relfree_k$, then
  there exists a group $H$ with $G\conv H$ and
  $\relfree_k=H/\variety(H)$:
  \[\begin{matrix}G & \textcolor{gray}{\conv} & \textcolor{gray}{H}\\
    \twoheaddownarrow && \textcolor{gray}{\twoheaddownarrow}\\
    G/\variety(G) &\conv & \relfree_k.\end{matrix}\]
\end{lemma}
\begin{proof}
  We proceed first as in the proof of Lemma~\ref{lem:quotients}, to
  construct a group $H$ and an epimorphism
  $\rho:H\twoheadrightarrow\relfree_k$.

  On the one hand, $\variety(H)\subseteq\ker\rho$, because
  $\relfree_k$ belongs to $\variety$. On the other hand, consider
  $c\in\ker\rho$, and write $c=w(T)$ as a word in the generators $T$
  of $H$. Then $\rho(w(T))=1$, so $w$ belongs to the variety
  $\variety(\free_k)$ because $\relfree_k$ is relatively free. It
  follows that $c$ belongs to $\variety(H)$.
\end{proof}

\begin{lemma}[\cite{champetier-guirardel:limitgroups}*{Proposition~2.25}]\label{lem:quotients2}
  If $G\conv H$ and $A$ is a quotient of $G$, then there exists a
  group $B$ with $A\conv B$ and $B$ is a quotient of $H$:
  \[\begin{matrix}G & \conv & H\\
    \twoheaddownarrow && \textcolor{gray}{\twoheaddownarrow}\\
    A &\textcolor{gray}{\conv} & \textcolor{gray}{B}.\end{matrix}\qedhere\]
\end{lemma}

Let us turn to the converse property: if $A\conv B$ and $B$ is a
subgroup of $H$, does there exist a group $G$ containing $A$ with
$G\conv H$?  Given a subgroup $B$ of a group $H$, we say that the pair
$(H,B)$ satisfies the ``\extsmallergroups'' property if, whenever $A$
is a group which \preforms\ $B$, there exists a group $G$ which
\preforms\ $H$ and contains $A$:
\[\begin{matrix}\textcolor{gray}{\exists G} & \textcolor{gray}{\conv} & H\\
  \textcolor{gray}{\cup} && \cup\\
  \forall A & \conv & B.\end{matrix}\]
We then say that $H$ has the ``\extsmallergroups'' property
if $(H,B)$ has that property for all finitely generated subgroups $B$
of $H$.

\begin{question}
  Which finitely generated groups have the ``\extsmallergroups'' property?
\end{question}

It is clear that if $H$ has very few subgroups, for example if every
proper subgroup of $H$ is finite, then $H$ has the
``\extsmallergroups'' property.

\begin{lemma}
  All finitely generated abelian group have the ``\extsmallergroups''
  property.
\end{lemma}
\begin{proof}
  Inclusions of finitely generated abelian groups into one another can
  be decomposed into the following ``elementary inclusions'': $B\subgp
  B\oplus\Z$, $B\subgp B\oplus\Z/a\Z$ and $B\oplus\Z/a\Z\subgp
  B\oplus\Z/ab\Z$. Similarly, the cases to consider for $A$ that
  \preforms\ $B$ are of the form $\Z\oplus\Z/ac\Z \conv
  \Z\oplus\Z/a\Z$ and $\Z\conv\Z^2$. To prove the lemma, it suffices
  therefore to consider the following case: $B=\Z^2\oplus\Z/a\Z$ is a
  subgroup of $H=\Z^2\oplus\Z/ab\Z$, and $A=\Z\oplus\Z/ac\Z$
  \preforms\ $B$.  We observe that in this case $G:=\Z\oplus\Z/abc\Z$
  contains $A$, and \preforms\ $H$.
\end{proof}

\begin{example}[Groups without the ``\extsmallergroups'' property]
  There are finitely generated groups $A\conv B\subgp H$ such that
  there exists no group $G$ with $A\subgp G\conv H$.

  Take indeed $A=\free_2\wr\Z$; it \preforms\ $B=\Z^2\wr\Z$, which
  is metabelian. By~\cite{baumslag:sgfpmetabelian}, every metabelian
  group imbeds in a finitely presented metabelian group $H$. If
  $G\conv H$, then $G$ is a quotient of $H$. This shows that every
  group which \preforms\ $H$ is metabelian. Therefore, there are no
  groups that \preform\ $H$ that contain $A$ as a subgroup.
\end{example} 

\begin{example}[Finitely presented groups without the ``\extsmallergroups'' property]
  Here is another example of this kind. Consider a finitely presented
  infinite torsion-free simple group $H$ containing a non-abelian free
  group $B=\free_3$ as a subgroup; such groups do exist,
  see~\cite{burger-m:fpsimple}.  Set $A=\free_ 2\times\Z/2\Z$; then
  $A\conv B$ and $B\subgp H$. However, if $G\conv H$, then $G=H$
  because $H$ is finitely presented and simple. However, $H$ does not
  contain $A$ because $H$ is torsion-free.
\end{example}

It is usually not true that, if $G$ \preforms\ $H$, then
the torsion of $G$ and $H$ coincide. Here is a partial result in this
direction:
\begin{lemma}\label{lem:comparingverbalsubgroups}
  Let $G$ and $H$ be groups with $H$ finitely presented and $G\conv
  H$, and let $\variety$ be a variety. Then
  \begin{enumerate}
  \item $\#\variety(G) = \#\variety(H)$;
  \item if $\variety(G)$ is finite, then $\variety(G)$ is isomorphic
    to $\variety(H)$.
  \end{enumerate}
\end{lemma}

\begin{proof}
  By Lemma~\ref{lem:identity}(1), the group $G$ is a quotient of $H$,
  so $\variety(G)$ is a quotient of $\variety(H)$. In particular,
  $\#\variety(G)\le\#\variety(H)$. Furthermore, if $\variety(H)$ is
  finite then Lemma~\ref{lem:identity}(3) implies that $\variety(G)$
  and $\variety(H)$ are isomorphic. It therefore remains to prove
  $\#\variety(G)\ge\#\variety(H)$. We will prove in fact that, if
  $\#\variety(H)\ge N$, then $\#\variety(G)\ge N$.

  Choose generating sets $S_n$ of $G$ and $T$ of $H$, of cardinality
  $k$, such that $(G,S_n)$ converges to $(H,T)$ is the space
  $\markedgroups$ of marked groups.

  Consider then $N$ distinct elements $h_1,\dots,h_N$ in
  $\variety(H)$, and write each $h_j=w_j(T)$ for a word
  $w_j\in\variety(\free_k)$. Take $R\in\N$ bigger than the length of
  each $w_j$, and let $i$ be such that the balls of radius $R$ in
  $\cay G{S_i}$ and $\cay HT$ coincide. Then the ball of radius $R$ in
  $\cay HT$ contains at least the $N$ distinct elements
  $h_1,\dots,h_N$ from $\variety(H)$, so the ball of radius $R$ in
  $\cay G{S_n}$ also contains at least $N$ distinct elements
  $w_1(S_n),\dots,w_N(S_n)$ from $\variety(G)$.
\end{proof}

\subsection{Universal theories of solvable groups}
For a group $G$, we denote by $G^{(n)}$ its derived series, with
$G^{(0)}=G$ and $G^{(n+1)}=[G^{(n)},G^{(n)}]$. In particular
$G^{(1)}=G'$ and $G^{(2)}=G''$.

Here is an example of metabelian group that \preforms\ the free
group in its variety. In the next sections, we will study when a
nilpotent group \preforms\ the free group in the variety it
generates.
\begin{example}
  We have $\Z\wr\Z\conv\free_2/\free_2''$.
\end{example}
\begin{proof}
  Consider the presentation $\langle a,t\mid [a,a^{t^m}]\forall
  m\rangle$ of $\Z\wr\Z$, and its generating sets
  $S_n=\{t,t^na\}$. Write $u=at^n$; then $[t,u]=[a,t]$, and
  $[t,u]^{t^xu^y}$ all have distinct supports, for $|x|,|y|\le
  n$.
\end{proof}

Chapuis considers in~\cite{chapuis:univsolvable} the universal theory
of some solvable groups; he shows that $\free_k/\free_k''$ and
$\Z^k\wr\Z^\ell$ have the same universal theory. An explicit
description of that theory is given in~\cite{chapuis:metabelian}. On
the other hand, $\Z\wr\Z\wr\Z$ and $\free_k/\free_k^{(3)}$ do not have
the same theory.

Timoshenko proves in~\cite{timoshenko:preserve} that, if $G_1,G_2$
have the same universal theory, and $H_1,H_2$ have the same universal
theory, then $G_1\wr H_1$ and $G_2\wr H_2$ have the same universal
theory. He shows, however, that the varietal wreath product does not,
in general, enjoy this property; in particular, it fails in the
metabelian variety~\cite{timoshenko:metabelian}.

He also shows in~\cite{timoshenko:uesolvable} that, if $G$ is the
quotient of $\mathbb S_{2,n}:=\free_2/\free_2^{(n)}$ by a finitely
generated normal subgroup, and has the same universal theory as
$\mathbb S_{2,n}$, then either $G\cong\mathbb S_{2,n}$ or $G$ is a
verbal wreath product $\Z\wr\Z$, in the variety of soluble groups of
class $n-1$. He shows:
\begin{lemma}
  Let $\relfree$ be a free group in a variety $\variety$, and let $H$
  be a subgroup of $\relfree$ that generates the same variety
  $\variety$. Assume that $\relfree$ is discriminating
  (see~\S\ref{ss:freenil}). Then the universal theories of $\relfree$
  and of $H$ coincide.\qedhere
\end{lemma}

\begin{lemma}
  Let $G,H$ be groups in a variety $\variety$, and assume that $G$ is
  universally equivalent to $H$. Then $A$ is residually $B$.\qedhere
\end{lemma}

Timoshenko also considered the universal theories of partially
commutative metabelian groups in~\cite{timoshenko:pcmetabelian} and
subsequent papers.

\section{Abelian groups}\label{ss:abelian}
By Corollary~\ref{cor:polyorder}, the relation $\conv$ is a partial order
on the set of abelian groups. The following is straighforward.
\begin{lemma}
  For non-zero $m,n\in\N$, we have $\Z^m\conv\Z^n$ if and only if
  $m\le n$.
\end{lemma}
\begin{proof}
  If $\Z^m\conv\Z^n$, then $\Z^m$ is a quotient of $\Z^n$ by
  Lemma~\ref{lem:fp}, so $m\le n$. Conversely, if $m\le n$, then
  choose for $\Z^n$ a basis $T$ as generating set, and let
  $\{e_1,\dots,e_m\}$ be a basis of $\Z^m$. For arbitrary $R\in\N$,
  choose $S=\{e_1,\dots,e_m,Re_1,R^2e_1,\dots,R^{n-m}e_1\}$ as
  generating set for $\Z^m$, and note that $\cay{\Z^m}S$ and
  $\cay{\Z^n}T$ agree on a ball of radius $R$.
\end{proof}

We now show that all infinite abelian groups are in the same component
of $\conv$, which has diameter $2$; more precisely,
\begin{proposition}\label{prop:net}
  The restriction of $\conv$ to infinite abelian subgroups is a
  \emph{net}: a partial order in which every pair of elements has an
  upper bound.
\end{proposition}

\begin{proposition}\label{prop:abelian}
  For infinite abelian finitely generated groups $A,B$, we have
  $A\conv B$ if and only if $A$ is a quotient of $B$ via a map
  $B\twoheadrightarrow A$ that is injective on the torsion of $B$.
\end{proposition}
\begin{proof}
  If $A\conv B$, then $A$ is a quotient of $B$ by
  Lemma~\ref{lem:fp}. Let $R$ be larger than the order of the torsion
  of $A$ and $B$, and let $S,T$ be generating sets of $A,B$
  respectively such that $\cay AS$ and $\cay BT$ coincide in the ball
  of radius $R$. Then all torsion elements of $B$ belong to that ball,
  and are mapped, by the identification of the ball, to torsion
  elements of $A$. This imbeds the torsion of $B$ into that of $A$.

  Conversely, consider an epimorphism $B\twoheadrightarrow A$ that is
  injective on the torsion of $B$. Let $B=G_0\twoheadrightarrow
  G_1\twoheadrightarrow\cdots\twoheadrightarrow G_n=A$ be a maximal
  sequence of non-invertible epimorphisms. If we prove $G_i\conv
  G_{i-1}$ for all $i=1,\dots,n$, then we have $A\conv B$ by
  Lemma~\ref{lem:preorder}, so we may restrict to a minimal
  epimorphism $\pi\colon B\twoheadrightarrow A$. Its kernel is thus infinite
  cyclic, and we have reduced to the case $A=\Z\oplus\Z/(k\ell)\Z$ and
  $B=\Z^2\oplus\Z/k\Z$.

  In that case, we consider $T=\{f_1,f_2,f_3\}$ the standard
  generating set for $B$, and denote by $\{e_1,e_2\}$ the standard
  generators for $A$. For arbitrary $R\in\N$, we consider the
  generating set $S=\{\ell e_1,e_2,e_1+Re_2\}$ for $A$, and note that
  the balls of radius $R$ in $\cay BT$ and $\cay AS$ coincide.
\end{proof}

\begin{proof}[Proof of Proposition~\ref{prop:net}]
  Consider $A,B$ abelian groups, written as
  \[A=\bigoplus_{i=1}^a\Z/m_i\Z,\qquad B=\bigoplus_{i=1}^b\Z/n_i\Z.\]
  Then both groups \preform\ $\Z^{\max(a,b)}$.
\end{proof}

\begin{corollary}\label{cor:tf=>lin}
  Let $A$ be an infinite abelian group. Then $A$ is torsion-free if
  and only if the set of groups that are \preformed\ by $A$ is
  linearly ordered.
\end{corollary}
\begin{proof}
  If $A=\Z^d$ and $A\conv B$, then $B=\Z^e$ for some $e\ge d$. The set
  of such $B$ is order-isomorphic to $\{d,d+1,\dots\}$.

  Now suppose that $A$ is not torsion-free. By
  Proposition~\ref{prop:abelian}, we have $A\conv\Z^d\oplus\Z/p\Z$ for
  some $p>1$ and $d>1$. Then $A\conv\Z^{d+1}$ and
  $A\conv\Z^{d+1}\oplus\Z/p\Z$, but these last groups are not comparable.
\end{proof}

Let us denote by $\mathscr A$ the subset of $\markedgroups$ consisting of
abelian groups, and by $\mathscr A/{\cong}$ the set of isomorphism
classes of abelian groups; as we noted above, $(\mathscr
A/{\cong},\conv)$ is a net.
\begin{corollary}
  Every finite partial order is imbeddable in $(\mathscr
  A/{\cong},\conv)$.
\end{corollary}
\begin{proof}
  Let $(X,\le)$ be a partially ordered set. We identify $x\in X$ with
  $I_x:=\{z\in X\colon z\ge x\}$, and have $I_y\subseteq
  I_x\Leftrightarrow x\le y$; therefore, we assume without loss of
  generality that $X$ is contained, for some $N\in\N$, in the
  partially ordered set of subsets of $\{1,\dots,N\}$, ordered under
  reverse inclusion.

  Consider $N$ distinct prime numbers $p_1,\dots,p_N$. For any subset
  $U\subseteq\{1,\dots,N\}$, consider the $N+1$-generated group $A_U$
  defined by
  \[A_U=\bigoplus_{i\in U}\Z/p_i\Z\oplus\Z^{1+N-\#U}.
  \]
  Observe that the torsion subgroup of $A_U$ is contained in the
  torsion group of $A_{U'}$ if and only if $U'\subseteq U$. Observe
  also that if $U'\subseteq U$, then $A_U$ is a quotient of
  $A_{U'}$. By Proposition~\ref{prop:abelian}, we get $A_U\conv
  A_{U'}$ if and only if $U'\subseteq U$.
\end{proof}

\begin{remark}
  Some countable orders cannot be imbedded in $(\mathscr
  A/{\cong},\conv)$; for example, $\N \cup\{\infty\}$. Observe indeed
  that a countable increasing sequence of non-isomorphic abelian
  groups has no common upper bound in $(\mathscr A/{\cong},\conv)$.
\end{remark}

\begin{proposition}\label{prop:aut_ab}
  The group of order-preserving bijections of $(\mathscr
  A/{\cong},\conv)$ is the infinite symmetric group on a countable
  set. If we identify this countable set with the prime numbers, then
  the action on infinite abelian groups is as follows. A permutation
  $p\mapsto \sigma(p)$ of the primes acts as
  \begin{equation}\label{eq:aut_ab}
    \Z^d\oplus\Z/p_1^{\nu_1}\Z\oplus\cdots\oplus\Z/p_k^{\nu_k}\Z\mapsto
    \Z^d\oplus\Z/\sigma(p_1)^{\nu_1}\Z\oplus\cdots\oplus\Z/\sigma(p_k)^{\nu_k}\Z.
  \end{equation}
\end{proposition}
\begin{proof}
  As a countable set, we take the set $\mathscr P$ of prime
  numbers. By Proposition~\ref{prop:abelian}, the group $\mathfrak S$
  of permutations of $\mathscr P$ acts on $(\mathscr A/{\cong},\conv)$
  by~\eqref{eq:aut_ab}. We wish to prove that there
  are no other order-preserving bijections. We implement this in the
  following lemmas.
  
  \begin{lemma}\label{lem:aut_ab:1}
    Every order-preserving bijection of infinite abelian groups fixes
    torsion-free abelian groups.
  \end{lemma}
  \begin{proof}
    By Corollary~\ref{cor:tf=>lin}, torsion-free abelian groups are
    characterized by the fact that the set of groups that they
    \preform\ is linearly ordered. Let $\phi$ be an order-preserving
    bijection. Observe that $\phi$ must fix the minimal element
    $\Z$. Note that groups that are \preformed\ by $\Z$ are linearly
    ordered by $\N$, so admit no order isomorphism. Therefore,
    $\phi(\Z^d)=\Z^d$ for any $d\ge 1$.
  \end{proof} 

  \begin{lemma}\label{lem:aut_ab:2}
    Every order-preserving bijection of infinite abelian groups
    preserves the number of factors in a minimal decomposition as a
    product of cyclic groups.
  \end{lemma}
  \begin{proof}
    Consider an infinite abelian group $A$, and let $\ell$ be the
    minimal number of cyclic subgroups in the decomposition of $A$ in
    a product of (finite or infinite) cyclic groups. Since $A$ is
    infinite, at least one subgroup in the decomposition is infinite.
    We know that for any $p\in\N$ the group $\Z+p\Z$ \preforms\ $\Z^2$,
    so $A$ \preforms\ $\Z^\ell$.

    Observe also that for $k<\ell$ the group $A$ cannot be generated
    by $k$ elements, so $A$ is not a quotient of $\Z^k$.  By
    Proposition \ref{prop:abelian}, $A$ doesn't \preform\ $\Z^k$ for
    $k<\ell$.

    Let $\phi$ be an order-preserving bijection. By
    Lemma~\ref{lem:aut_ab:1}, we have $\phi(\Z^k)=\Z^k$ for all
    $k\ge1$, so $\phi(A)$ \preforms\ $\Z^\ell$ but not $\Z^k$ for
    $k<\ell$. Therefore, $\phi(A)$ requires precisely $\ell$ factors
    in a minimal decomposition as a product of cyclic groups.
  \end{proof}
  
  \begin{lemma}\label{lem:aut_ab:3}
    Every order-preserving bijection $\phi$ of infinite abelian groups
    preserves the number of finite and infinite factors in a minimal
    decomposition as a product of cyclic groups.
  \end{lemma}  
  \begin{proof}
    Let $A$ be an infinite abelian group. Let $t$ be the minimal
    number of finite cyclic groups in its decomposition into a product
    of cyclic ones, and let $t+d$ be the minimal total number of
    finite cyclic groups in such decomposition.  We have $A=\Z^d
    \oplus \bigoplus_{i=1}^t \Z/ n_i \Z$, with $n_i\ge 2$.  Observe
    that $A$ is \preformed\ by $\Z \oplus \bigoplus_{i=1}^t \Z/ n_i
    \Z$, and thus is \preformed\ by some group whose minimal total
    number of cyclic groups in a decomposition equals $t+1$. Observe
    then that $A$ is not \preformed\ by any group for which this
    minimal number is $\le t$. Indeed, if $B$ \preforms\ $A$, then $B$
    is an infinite group, so the number of infinite cyclic group in
    the decomposition is $\ge 1$.  We know that the torsion subgroup
    of $A$, that is $\bigoplus_{i=1}^t \Z/ n_i \Z$, is a subgroup of
    the torsion subgroup of $B$. Therefore, the minimal number of
    finite cyclic groups in the decomposition of $B$ is at least
    $t$. The statement of the lemma now follows from the previous
    lemma.
  \end{proof}
  
  Consider now an order-preserving bijection $\phi$ of abelian groups.
  Let us show that for every prime $p$ there exists a prime $q$ such
  that $\phi(\Z\oplus\Z/p\Z)= \Z\oplus\Z/q\Z$.  First observe that any
  group with non-trivial torsion and which \preforms\ $\Z^2$ has the
  form $\Z\oplus\Z/n\Z$ for some $n\ge 2$.  If $n$ is not a prime
  number, then $n$ can be written as $n=n_1 n_2$ with $n_1, n_2 \ge 2$
  and in this case $\Z\oplus\Z/n\Z$ \preforms\ $\Z\oplus\Z/n_1\Z$.  This
  implies that the groups of the form $A=\Z\oplus\Z/p\Z$ are
  characterized by the following properties: $A$ is not torsion-free;
  $A$ \preforms\ $\Z^2$; if $B$ is such that $A\conv B\conv\Z^2$ then
  either $B=A$ or $B=\Z^2$. This implies that $\phi(\Z\oplus\Z/p\Z)$
  is isomorphic to $\Z\oplus\Z/q\Z$ for some prime $q$.
  
  As we have already mentioned, every permutation of the primes
  induces an order-preserving permutation of infinite abelian
  group. It remains to prove that a permutation of infinite abelian
  groups is determined by its action on groups of the form
  $\Z\oplus\Z/p\Z$. Consider therefore such a permutation $\phi$, and
  assume that it fixes $\Z\oplus\Z/p\Z$ for all $p\in\mathscr P$. We
  wish to show that it fixes every abelian group.
  
  \begin{lemma}\label{lem:aut_ab:4}
    Let $\phi$ be an order-preserving bijection of the infinite
    abelian groups, such that $\phi(\Z\oplus\Z/p\Z)=\Z\oplus\Z/p\Z$
    for all primes $p$.

    Then for all $k,m\ge1$ we have $\phi(\Z^k\oplus\Z/p^m\Z) =
    \Z^k\oplus\Z/p^m\Z$.
  \end{lemma}
  \begin{proof}
    Set $A=\Z^k\oplus\Z/p^m\Z$. By Lemma~\ref{lem:aut_ab:3}, we have
    $\phi(A)=\Z^k\oplus\Z/n\Z$ for some $n\ge2$. We proceed by
    induction on $m$ to show that $A$ is fixed.

    If $m=1$, then $A$ is \preformed\ by $\Z\oplus\Z/p\Z$ which is
    fixed, so $\phi(A)$ is also \preformed\ by this group, and
    $n|p$. Since $n\neq1$, we have $n=p$ as required.

    Consider then $m\ge2$. We have $A\conv\Z^{k+1}\oplus\Z/p^{m-1}\Z$,
    which is fixed by induction, so $p^{m-1}|n$, and in fact
    $p^{m-1}\neq n$ because $\phi(A)$ does not belong to the set of
    groups of the form $\Z^\ell\oplus\Z/p^{m-1}$ which are all fixed
    by $\phi$.

    On the other hand, $A$ doesn't \preform\ any of the groups
    $\Z^\ell\oplus\Z/q\Z$ for $q\neq p$ prime, which are fixed, so
    $\phi(A)$ neither \preform\ any of these groups, and $n=p^e$
    for some $e\ge m$.

    Now there are precisely $m+1$ groups between $A$ and $\Z^{k+2}$,
    namely all $\Z^{k+1}\oplus\Z/p^i\Z$ for $i=0,\dots,m$. This
    feature distinguishes $A$ from $\Z\oplus\Z/p^e\Z$ for all $e\neq
    m$, and therefore $A$ is fixed by $\phi$.
  \end{proof}
      
  \begin{lemma}\label{lem:aut_ab:5}
    Let $\phi$ be an order-preserving bijection of the infinite
    abelian groups, such that $\phi(\Z\oplus\Z/p\Z)=\Z\oplus\Z/p\Z$
    for all primes $p$.

    Then $\phi$ fixes all groups of the form $\Z^k\oplus C$ with $C$
    an abelian $p$-group.
  \end{lemma}
  \begin{proof}
    By Lemma~\ref{lem:aut_ab:3}, we have $\phi(\Z^k\oplus
    C)=\Z^k\oplus C'$ for a finite group $C'$ with the same number of
    factors in a minimal decomposition as a product of cyclic groups.

    Write $C=\bigoplus_{i=1}^r\Z/p^{e_i}\Z$, with $1\le e_1\le
    e_2\le\dots\le e_r$. We proceed by induction on $r$, the case
    $r=1$ being covered by Lemma~\ref{lem:aut_ab:4}.

    Write $A=\Z^k\oplus C$. Since, when $\ell$ is large,
    $A\conv\Z^\ell\oplus\Z/q\Z$ with $q$ prime if and only if $q=p$,
    we find that $C'$ is a $p$-group of the form
    $\bigoplus_{i=1}^r\Z/p^{f_i}\Z$, with $1\le f_1\le\dots\le f_r$.

    Consider $B=\Z^{k+1}\oplus\bigoplus_{i=1}^{k-1}\Z/p^{e_i}\Z$,
    which is fixed by induction. We have $A\conv B$, so $\phi(A)\conv
    B$ and therefore $f_1=e_1,\dots,f_{r-1}=e_{r-1},f_r\ge e_r$ by
    Proposition~\ref{prop:abelian}. It remains to prove $f_r=e_r$.

    Again by induction, the group $\Z\oplus B$ is fixed by
    $\phi$. There are $e_r+1$ groups between $A$ and $\Z\oplus B$,
    namely $B\oplus\Z/p^e\Z$ for $e=0,\dots,e_r$. This distinguishes
    $A$ among all
    $\Z^k\oplus\bigoplus_{i=1}^{r-1}\Z/p^{e_i}\Z\oplus\Z/p^{f_r}\Z$
    with $f_r\ge e_r$.
  \end{proof}
  
  We are ready to finish the proof of
  Proposition~\ref{prop:aut_ab}. Consider again $\phi$ fixing all
  $\Z\oplus\Z/p\Z$ for $p$ prime, and an abelian group $A=\Z^k\oplus
  C$ with $C$ finite; let us show that the torsion of $\phi(A)$ is
  isomorphic to $C$.

  First, by Lemma~\ref{lem:aut_ab:4}, we have $\phi(A)=\Z^k\oplus C'$
  for a finite group $C'$. Observe that, for $\ell$ large and $D$ a
  $p$-group, $A$ \preforms\ $\Z^\ell\oplus D$ if and only if $D$ is a
  subgroup of $C$. By Lemma~\ref{lem:aut_ab:5}, this group
  $\Z^\ell\oplus D$ is fixed by $\phi$, so $C$ and $C'$ have the same
  $p$-subgroups. Since every abelian group is the product of its
  $p$-Sylow subgroups, it follows that $C$ and $C'$ are isomorphic.
\end{proof}

\subsection{Virtually abelian groups}
There are countably many components of virtually abelian groups, as we
now show:

\begin{example}\label{ex:component22p}
  Let $N_{2,2}$ be the group with presentation
  \[N_{2,2}=\langle a,b\mid c=[a,b]\text{ central}\rangle,
  \]
  and every $n\in\N$, let $G_n$ be the virtually abelian group
  \[N_{2,2,n}=N_{2,2}/\langle c^n\rangle=\langle a,b\mid [a,b]^n,
  [a,b]\text{ central}\rangle.
  \]
  Then every $N_{2,2,n}$ is virtually $\Z^2$, but if $m\neq n$ then
  $N_{2,2,n}$ and $N_{2,2,m}$ belong to different components of
  $\markedgroups/{\cong}$.
\end{example}
\begin{proof}
  Without loss of generality, assume $m<n$, and let $H$ belong to the
  component of $N_{2,2,m}$; so there is a sequence
  $N_{2,2,m}=H_0,H_1,\dots,H_\ell=H$ with $H_i\conv H_{i-1}$ or
  $H_{i-1}\conv H_i$ for all $i=1,\dots,\ell$. By
  Lemma~\ref{lem:identity}(1,2), every $H_i$ is finitely presented and
  satisfies the identity $[x,y]^m$. However, $N_{2,2,n}$ does not
  satisfy this identity.
\end{proof}

\begin{remark}\label{rem:successor22p}
  If $p$ is prime, then the set of groups limit greater than
  $N_{2,2,p}$ is precisely
  $\{N_{2,2,p}\times\Z^\ell\colon\ell\in\N\}$.
\end{remark}
\begin{proof}
  Elements of $N_{2,2,p}$ may uniquely be written in the form
  $a^xb^yc^z$ for some $x,y\in\Z$ and $z\in\{0,\dots,p-1\}$. Consider
  a sequence of generating sets $S_1,S_2,\dots$ of same cardinality
  $k$. Clearly, if each $S_n$ is changed by a bounded number of
  Nielsen transformations, then without loss of generality one may
  assume (up to taking a subsequence) that the same transformations
  are applied to all $S_n$, and therefore the limit does not change.

  Using at most $pk$ transformations, the set $S_n$, whose elements we
  write as $\{s_{n,1},\dots,s_{n,k}\}$, can be transformed in such a
  manner that two elements $s_{n,1},s_{n,2}$ generate $N_{2,2,p}$
  while the other $s_{n,3},\dots,s_{n,k}$ are of the form $a^xb^yc^z$
  with $p|x$ and $p|y$, and therefore belong to the centre of
  $N_{2,2,p}$. Some of these elements will belong to $\langle
  s_1,s_2\rangle$ in the limit, and others will generate extra abelian
  factors.
\end{proof}

\section{Nilpotent groups}
Given a group $G$, we denote its lower central series by
$\gamma_1(G)=G$ and $\gamma_{i+1}(G)= [G, \gamma_i(G)]$ for all
$i\ge1$. By $N_{s,k}=\free_k/\gamma_{s+1}(\free_k)$ we denote the free
nilpotent group of class $s$ on $k$ generators.

We study in this section the structure of connected components of
nilpotent groups; our main result is that, if $G/\Torsion(G)$
generates the same variety as $G$, then the connected component of $G$
is determined by the variety that it generates and conversely.

\subsection{Free groups and subgroups in nilpotent varieties}\label{ss:freenil}
Following~\cite{neumann:varieties}*{Definition~17.12}, a group $G$ is
said to be \emph{discriminating} if, given any finite set $\mathscr W$
of identities that do not hold in $G$ (i.e., for every $w\in\mathscr
W$ there are $g_1,g_2,\dots\in G$ with $w(g_1,\dots)\neq1$), all
identities can be falsified simultaneously (i.e.\ there are
$g_1,g_2,\dots\in G$ such that $w(g_1,\dots)\neq1$ for all
$w\in\mathscr W$). We will say $G$ is \emph{discriminating on $k$
  generators} if, given any finite set $\mathscr W$ of identities in
$k$ letters that do not hold in $G$ (i.e., for every $w\in\mathscr W$
there are $g_1,\dots,g_k\in G$ with $w(g_1,\dots,g_k)\neq1$), all
identities can be falsified simultaneously on a generating set (i.e.\
there exists a generating set $\{g_1,\dots,g_k\}$ of $G$ such that
$w(g_1,\dots,g_k)\neq1$ for all $w\in\mathscr W$).

Baumslag, Neumann, Neumann, and Neumann show
in~\cite{baumslag-n-n-n:fgvarieties}*{Corollary~2.17} that finitely
generated torsion-free nilpotent groups are discriminating; see
also~\cite{neumann:varieties}*{Theorem~17.9}. If $G$ is a nilpotent
group with torsion, the matter is more delicate: Bausmlag and Neumanns
prove in the same place that $G$ is discriminating if and only if $G$
and $G/\Torsion(G)$ generate the same variety.

\begin{lemma}\label{freesubgroupsvarieties}
  Let $G$ be a discriminating group, and let $\variety$ be the variety
  generated by $G$. Let $\relfree_k:=\free_k/\variety(\free_k)$ be the
  free group on $k$ generators in $\variety$. Then for every
  $k\in\N$ there exists a group $H$ that is \preformed\ by $G$ and
  contains $\relfree_k$ as a subgroup.

  If furthermore $G$ is discriminating on $k$ generators, then $G$
  \preforms\ $\relfree_k$.
\end{lemma}
\begin{proof}
  Consider first a finite set of words $\mathscr W\subset\free_k$ that
  are not identities of $\relfree_k$, that is $w(s_1,\dots,s_k)\neq1$
  in $\relfree_k$ for all $w\in\mathscr W$, with $\{s_1,\dots,s_k\}$ a
  free generating set for $\relfree_k$. Observe that, for each
  $w\in\mathscr W$, there exist elements $g_{w,1},\dots,g_{w,k}\in G$
  with $w(g_{w,1},\dots,g_{w,k})\neq1$; otherwise, $w$ would be an
  identity in $G$ and therefore would vanish on $\relfree_k$. Since
  $G$ is discriminating, there exist $g_{\mathscr
    W,1},\dots,g_{\mathscr W,k}\in G$ such that $w(g_{\mathscr
    W,1},\dots,g_{\mathscr W,k})\neq1$ for all $w\in\mathscr W$.

  We apply this with $\mathscr W$ the set of words of length at most
  $R$ in $\free_k$ that are not identities in $\relfree_k$, and denote
  the resulting $g_{\mathscr W,1},\dots,g_{\mathscr W,k}$ by
  $g_{R,1},\dots,g_{R,k}$.

  Let $S$ be a finite generating set for $G$, and put
  $S_R=S\sqcup\{g_{R,1},\dots,g_{R,k}\}$. Choose an accumulation point
  $(H,T)$ of the sequence $(G,S_R)$ in the space $\markedgroups$ of
  marked groups. Then $H$ contains $\relfree_k$ as the subgroup
  generated by the limit of $\{g_{R,1},\dots,g_{R,k}\}$.

  If $G$ is discriminating on $k$ generators, then we can take
  $S=\emptyset$ in the previous paragraph, to see that $H$ is
  isomorphic to the relatively free group $\relfree_k$.
\end{proof}

For a real constant $C$, let us say that the sequence of positive real
numbers $x_1,x_2,\dots,x_s$ grows \emph{at speed $C$} if $x_1\ge C$
and $x_{i+1}\ge x_i^C$ for $i=1,\dots,s-1$. Similarly, an unordered
set $\{x_1,\dots,x_s\}$ grows \emph{at speed $C$} if it admits an
ordering that grows at speed $C$.

\begin{lemma}\label{polynomials}
  Suppose that $f_1,\dots,f_t$ are nonzero polynomials in $s$
  variables with real coefficients. Then there exists $C$ such that
  $f_i(x_1, \dots, x_s)\ne0$ for all $i=1,\dots,t$ whenever
  $(x_1,\dots,x_s)$ grows at speed $C$.
\end{lemma}
\begin{proof}
  It suffices to prove the statement for a single polynomial $f$. Let
  $x_1^{e_1}\cdots x_s^{e_s}$ be the lexicographically largest
  monomial in $f$; namely, $e_s$ is maximal among all monomials in
  $f$; then $e_{s-1}$ is maximal among monomials of degree $e_S$ in
  $x_s$; etc. Then this monomial dominates $f$ as $(x_1,\dots,x_s)$
  grows faster and faster.
\end{proof}

\begin{lemma}\label{rapidgenerators}
  Consider $d\ge 1$. Then for all $e\ge d+1$ and all $C>0$ there
  exists a set of numbers
  $\{x_{1,1},x_{1,2},\dots,x_{1,d},x_{2,1},\dots,x_{e,1},x_{e,d}\}$
  growing at speed $C$ and such that
  $\{(x_{1,1},\dots,x_{1,d}),\dots,(x_{e,1},\dots,x_{e,d})$ is a
  generating set for $\Z^d$.
\end{lemma}
\begin{proof}
  It suffices to prove the statement for $e=d+1$.  We start by proving
  the following claim by induction on $n=1,\dots,d$: there exists an
  $n\times n$ integer matrix $(x_{i,j})$ whose coefficients grow at
  speed $C$, and such that for every $k=1,\dots,n$ the determinant of
  the upper left corner $(x_{i,j}\colon 1\le i,j\le k)$ is a prime
  number $p_k$, with all primes $p_1,\dots,p_n$ distinct.

  The induction starts by setting $x_{1,1}=p_1$ for some prime number
  $p_1>C$.

  Assume then that an $(n-1)\times(n-1)$ matrix $A_{n-1}=(x_{i,j})$
  has been constructed, with entries growing at speed $C$ and
  determinant a prime number $p_{n-1}$.

  First, an $n$th row $(x_{n,1},\dots,x_{n,n-1})$ may be added to
  $A_{n-1}$ in such a manner that the entries still grow at speed $C$,
  and the determinant $d_n$ of $A'_{n-1}=(x_{i,j}\colon i\neq n-1)$ is
  coprime to $p_{n-1}$. Indeed the coefficients
  $x_{n,1},\dots,x_{n,n-2}$ may be chosen arbitrarily as long as they
  grow fast enough. Then increasing $x_{n,n-1}$ increases the
  determinant of $A'_{n-1}$ by $p_{n-2}$ which is coprime to
  $p_{n-1}$; and sufficiently increasing this coefficient makes the
  augmented matrix $A''_{n-1}=(x_{i,j}\colon i\le n)$ still growing at speed
  $C$.

  Then an $n$th column may be added to $A''_{n-1}$ as follows. Start
  by choosing $x_{1,n},\dots,x_{n-2,n}$ arbitrarily as long as they
  grow fast enough, without fixing $x_{n-1,n}$ and $x_{n,n}$ yet. Call
  $A_n$ the resulting matrix. Then increasing $x_{n-1,n}$ decreases
  the determinant of $A_n$ by $d_n$, while increasing $x_{n,n}$
  increases the determinant of $A_n$ by $p_{n-1}$. Since $d_n$ and
  $p_{n-1}$ are coprime, there exist choices of $x_{n-1,n}$ and
  $x_{n,n}$ such that $A_n$ has determinant $1$; and the entries of
  $A_n$ grow at speed $C$, except perhaps for $x_{n,n}$.

  Now, by Dirichlet's theorem, there exists arbitrarily large primes
  $p_n$ that are $\equiv 1\pmod{p_{n-1}}$. For such a prime
  $p_n=1+ap_{n-1}$, add $a$ to the entry $x_{n,n}$ yielding a matrix
  $A_n$ of determinant $p_n$. Choosing $a$ large enough makes the
  coefficients of $A_n$ grow at speed $C$.

  To prove the lemma, consider a $d\times d$ matrix $A$ with integer
  entries growing at speed $C$ and determinant $p$. Its rows generate
  a subgroup of $\Z^d$ of prime index, and a single extra generator,
  with fast growing entries that are coprime to $p$, gives the desired
  generating set.
\end{proof}

We are ready to
sharpen~\cite{baumslag-n-n-n:fgvarieties}*{Corollary~2.17}, claiming
that torsion-free nilpotent groups are discriminating:
\begin{lemma}\label{lem:discriminatedbygenerators}
  Let $G$ be a torsion-free $k$-generated nilpotent group. Then, for
  each $N>k$, the group $G$ is discriminating on $N$ generators.
\end{lemma}
\begin{proof}
  We start by considering more generally poly-$\Z$ groups, namely
  groups $G$ admitting a sequence of subgroups $G=G_1\triangleright
  G_2\triangleright\dots\triangleright G_{\ell+1}=1$ such that
  $G_i/G_{i+1}\cong\Z$ for all $i$.

  If $G$ is torsion-free nilpotent and $(Z_i)$ denotes its ascending
  central series (defined inductively by $Z_0=1$ and
  $Z_{i+1}/Z_i=Z(G/Z_i)$), then each $Z_{i+1}/Z_i$ is free abelian, so
  the ascending central series can be refined to a series in which
  successive quotients are $\Z$.

  Choose for all $i=1,\dots,\ell$ a generator of $G_i/G_{i+1}$, and
  lift to an element $u_i\in G_i$. Then every $g\in G$ may uniquely be
  written in the form $g=u_1^{\xi_1}\cdots u_\ell^{\xi_\ell}$, and the
  integers $\xi_1,\dots,xi_\ell$ determine the element $g$, which we
  write $u^\xi$. Philip Hall proved
  in~\cite{hall:notesnilpotent}*{Theorem~6.5} that products and
  inverses are given by polynomials, in the sense that if $u^\xi
  u^\eta=u^\zeta$ and $(u^\xi){-1}=u^\chi$, then $\zeta_i$ and
  $\chi_i$ are polynomials in
  $\{\xi_1,\dots,\xi_\ell,\eta_1,\dots,\eta_\ell\}$ and
  $\{\xi_1,\dots,\xi_\ell\}$ respectively.  In particular, every
  identity $w\in\mathscr W$, in $N$ variables, is a polynomial in the
  exponents $\xi_{1,1},\dots,\xi_{\ell,N}$ of its arguments
  $x_1,\dots,x_N$ written as $u^{\xi_1},\dots,u^{\xi_N}$.

  By Lemma~\ref{rapidgenerators}, there exist sequences with
  arbitrarily fast growth that generate the abelianization of $G$; and
  by Lemma~\ref{polynomials} the identities in $\mathscr W$ will not
  vanish on these generators, if their growth is fast enough. Finally,
  since $G$ is nilpotent, a sequence of elements generates $G$ if and
  only if it generates its abelianization.
\end{proof}

\begin{lemma}\label{lem:nilprecedestfnil}
  Let $G$ be a finitely generated nilpotent group such that $G$ and
  $G/\Torsion(G)$ generate the same variety. Then $G$ \preforms\ a
  torsion-free nilpotent group.
\end{lemma}
\begin{proof}
  Infinite, finitely generated nilpotent groups have infinite
  abelianization; we apply Lemma~\ref{lem:verbalquotients} to $G$ and
  the variety $\variety$ of abelian groups. Since every infinite
  abelian group \preforms\ a free abelian group, we assume
  without loss of generality that $G$ has torsion-free abelianization.

  Assume that $G$ is $k$-generated, and consider $N>k$ and $R>0$.
  Consider the set $\mathscr W(R)$ of all words $w$ of length at most
  $R$ in $N$ variables such that, for some $g_1,\dots,g_N\in G$, the
  evaluation $w(g_1,\dots,g_N)$ is a non-trivial torsion element in
  $G$. In particular, such $w$ are not identities in $G$. Since $G$
  and $G/\Torsion(G)$ generate the same variety, none of these words
  is an identity in $G/\Torsion(G)$.  Since $G/\Torsion(G)$ is a
  torsion-free nilpotent group, Lemma~\ref{lem:discriminatedbygenerators}
  implies that $\mathscr W(R)$ is discriminated by an $N$-element
  generating set of $G/\Torsion(G)$, which we denote by $S'_R$. Let
  $S_R$ denote a preimage in $G$ of $S'_R$.  Since the abelianization
  of $G$ is torsion-free, it is isomorphic (under the natural quotient
  map) to the abelianization of $G/\Torsion(G)$. Therefore, $S_R$
  generates the abelianization of $G$, so generates $G$.
 
  Let $(H,T)$ be an accumulation point of the sequence $(G,S_R)$ in
  the space $\markedgroups$ of marked groups. Observe that $H$ is
  torsion-free. Indeed, by Lemma~\ref{lem:identity}(3) the torsion
  of $H$ imbeds in that of $G$; and if $a$ is a torsion element of
  $G$, then for all $R$ large enough there are words $w\in\mathscr
  W(R)$ that assume the value $a$. By construction of $S_R$, the value
  $a$ is not taken by a word of length $\le R$ in $S_R$, so $a$ does
  not have a limit in $H$.
\end{proof}

\begin{proposition}\label{mainpropositionnilpotent}
  Let $G$ be a $k$-generated nilpotent group, and assume that $G$ and
  $G/\Torsion(G)$ generate the same variety, $\variety$.

  Then, for every $N>k$, the group $G$ \preforms\
  $\relfree_N$.

  Consequently, the connected component of $G$ for the relation
  $\conv$ has diameter $2$.
\end{proposition}
\begin{proof}
  By Lemma~\ref{lem:nilprecedestfnil}, we may assume that $G$ is
  torsion-free nilpotent. By
  Lemma~\ref{lem:discriminatedbygenerators}, the group $G$ is
  discriminating on $N$ generators. By
  Lemma~\ref{freesubgroupsvarieties}, the group $G$ precedes
  $\relfree_N$.
\end{proof}

\begin{remark}
  The assumption that $G$ is torsion-free is essential for the first
  claim of the proposition above. Consider indeed the variety of
  nilpotent groups of nilpotent class $2$ in which every commutator is
  of order $p$. This variety is generated, e.g., by the group
  $N_{2,2,p}$ from Example~\ref{ex:component22p}. However, there does
  not even exist any group \preformed\ by $G$ and containing
  $\relfree_3$ as a subgroup, because the torsion $\relfree_3$ is
  larger than the torsion in $N_{2,2,p}$.
\end{remark}

\begin{remark} 
  Let $\variety$ be a nilpotent variety. Then, if
  $\relfree_m\conv\relfree_n$, then $m\le n$.
\end{remark}
\begin{proof}
  Since $\relfree_n$ is finitely presented, $\relfree_m$ is a quotient
  of $\relfree_n$. The abelianization of $\relfree_n$ is
  $n$-generated, so the abelianization of any quotient of $\relfree_n$
  is also $n$-generated, so $m\le n$.
\end{proof}

\noindent Proposition~\ref{mainpropositionnilpotent} has the following
\begin{corollary}
  Consider a nilpotent variety $\variety$ generated by a group $G$
  such  that $G/Torsion(G)$ also generates
  $\variety$. Let $c$ be the nilpotency class of $G$.

  For $m,n>c$, we have $\relfree_m\conv\relfree_n$ if and only if
  $m\le n$.
\end{corollary}
\begin{proof}
  It is known from~\cite{neumann:varieties}*{Theorem~35.11} that
  $\relfree_m$ generates $\variety$ as soon as $m\ge c$.
\end{proof}

\begin{remark}
  Consider a nilpotent variety $\variety$ generated by a torsion-free
  nilpotent group. For small $m,n$, the free groups $\relfree_m$ and
  $\relfree_n$ need not belong to the same component. For example, if
  $\variety$ the variety of nilpotent groups of class $5$, then
  $\relfree_2$ does not generate $\variety$, since it is metabelian
  but $\relfree_3$ is not. See~\cite{neumann:varieties}*{35.33} for
  details.
\end{remark}

\subsection{When generators of a variety lie in different components}
We will see that, if $G$ and $G/\Torsion(G)$ lie in different
varieties, then the variety of $G$ contains infinitely many connected
components under $\conv$.

\begin{lemma}\label{finiteverbal}
  Let $G$ be a nilpotent group such that $G$ and $G/\Torsion(G)$
  generate different varieties. There exists a variety $\variety$ such
  that the verbal subgroup $\variety(G)$ is non-trivial and finite.
\end{lemma}
\begin{proof}
  First recall that torsion elements of a nilpotent group $G$ form a
  finite subgroup of $G$.  Since $G$ and $G/\Torsion(G)$ generate
  different varieties, there exists an identity $w$ of $G/\Torsion(G)$
  that is not an identity in $G$. Set $\variety=\{w\}$; then
  $\variety(G)$ is non-trivial and is contained in the torsion of $G$,
  hence finite.
\end{proof}

\begin{corollary} \label{cor:nilpotentcomponents}
  Let $G$ be a nilpotent group and let $\variety$ be the variety that
  it generates. The connected component of $G$ coincides with the set
  of groups generating $\variety$ if and only if $G/\Torsion(G)$
  generates $\variety$. If this is not the case, the set of groups
  generating $\variety$ consists of infinitely many connected
  components for the relation $\conv$.
\end{corollary}
\begin{proof}
  If $G/\Torsion(G)$ generates $\variety$, the corollary follows from
  Proposition~\ref{mainpropositionnilpotent}. Assume now that
  $G/\Torsion(G)$ does not generate $\variety$. Then by
  Lemma~\ref{finiteverbal} there exists a variety $\mathcal W$ such
  that the verbal subgroup $\mathcal W(G)$ is non-trivial and finite.
  Observe that a verbal subgroup of a direct product is the product of
  itsx verbal subgroups.  Therefore, for all $n\in\N$, the verbal
  subgroups $\mathcal W(\bigtimes_n G)$ are non-isomorphic. By
  Lemma~\ref{lem:comparingverbalsubgroups}, all the groups $\bigtimes_n G$
  lie in distinct connected components. However, they all generate
  $\variety$.
\end{proof}

\subsection{Examples and illustrations}
In the variety of abelian groups, the following is true: if $G$ is a
quotient of $H$ and the torsion of $H$ imbeds in the torsion of $G$
under the quotient map, then $G\conv H$. This is not true anymore
among nilpotent groups.

\begin{example}\label{ex:n22xn22}
  Consider the groups $G=N_{2,2}$ and $H=N_{2,2}\times N_{2,2}$, see
  Example~\ref{ex:component22p}. Then both $G$ and $H$ are
  torsion-free, and $G$ is a quotient of $H$. However, $G$
  doesn't \preform\ $H$.
\end{example}
\begin{proof}
  Consider the following universal statement:
  \[\forall
  a,b,c,z(([a,b]=1\wedge[a,c]=1\wedge[b,c]\neq1)\Rightarrow[a,z]=1).\]
  It states that if $a$ commutes with two non-commuting elements $b$
  and $c$, then $a$ is central.

  This property does not hold in $H$: take $a,z$ the generators of the
  first $N_{2,2}$ and $b,c$ the generators of the second one.

  On the other hand, in $N_{2,2}$, this property holds. Indeed if
  $[a,b]=1$ then the image of $\{a,b\}$ in
  $N_{2,2}/Z(N_{2,2})\cong\Z^2$ lies in a cyclic subgroup; Similarly
  the image of $\{a,c\}$ lies in a cyclic subgroup; so either $a$ is
  central or the image of $\{b,c\}$ lies in a cyclic subgroup.
\end{proof}


\begin{example}
  As soon as the nilpotency class is allowed to grow beyond $4$, there
  exist nilpotent varieties whose free groups are not virtually free
  nilpotent. For example, consider the group
  $G=\free_3/\langle\free_3'',\gamma_5(\free_3)\rangle$. This group is
  nilpotent of class $4$, and is an iterated central extension of $29$
  copies of $\Z$. The $3$-generated free nilpotent groups of class $3$
  and $4$ have respectively $14$ and $32$ cyclic factors, so $G$ is
  not commensurable to either. This is easily seen in the (Malcev) Lie
  algebra associated with these groups.
\end{example}

\begin{lemma}
  Let $G$ be a non-virtually abelian nilpotent group. Then the
  connected component of $G$ is not isomorphic, as partially ordered
  set, to the component of abelian groups.
\end{lemma}
\begin{proof}
  In the component of abelian groups, the following holds: for any $A$
  there exists $B$ with $A\conv B$ and such that the set of groups
  that are \preformed\ by $B$ is linearly ordered.  We claim that
  the connected component of $G$ does not have this property.

  More precisely, for any non-virtually abelian nilpotent $G$, we
  construct incomparable groups $H_1,H_2$ that are both \preformed\ by
  $G$.

  Since $G$ is not virtually abelian, $[G,G]$ is infinite. Then both
  $G$ and $[G,G]$ have infinite abelianization, so that $G$ maps onto
  $N_{2,2}$, the free nilpotent group of class $2$ on $2$
  generators. Since $N_{2,2}\conv N_{2,k}$ for all $k\ge2$, there
  exists by Lemma~\ref{lem:quotients} a group $H_1$ such that
  $\gamma_2(H_1)/\gamma_3(H_1)$ has arbitrarily large rank, in
  particular rank larger than that of $\gamma_2(G)/\gamma_3(G)$. Set
  then $H_2=G\times\Z^d$ for $d$ larger than the rank of
  $H_1/\gamma_2(H_1)$. Then $H_1$ is not a quotient of $H_2$, because
  $\gamma_2(H_1)/\gamma_3(H_1)$ is not a quotient of
  $\gamma_2(H_2)/\gamma_3(H_2)$; and $H_2$ is not a quotient of $H_1$,
  because $H_2/\gamma_2(H_2)$ is not a quotient of
  $H_2/\gamma_2(H_2)$.
\end{proof}

\section{Imbeddability of orders. Solvable groups}

We characterize the preorders (transitive, reflexive relations) that
can be imbedded in the preorder of groups up to isomorphism, under the
relation $\conv$. We show in this manner that $\conv$ has a rich
structure, even when restricted to solvable groups of class $3$.

In this section, we view $\conv$ as a preorder on $\markedgroups$,
defined by $(G,S)\conv(H,T)$ if and only if $G\conv H$. For $X$ a set,
we denote by $\mathcal P(X)$ the family of subsets of $X$.
\begin{proposition}\label{prop:order}
  Let $\mathscr B$ be a countably infinite set, and let $\mathscr X$
  have the cardinality of the continuum. Put on $\mathcal P(\mathscr
  B)\times\mathscr X$ the preorder
  \[(X,c)\precsim(Y,c')\text{ if and only if }X\supseteq Y.
  \]
  Then the preorders $(\markedgroups,\conv)$ and $(\mathcal P(\mathscr
  B)\times\mathscr X,\precsim)$ imbed into each other.
\end{proposition}

We note that $(\mathcal P(\mathscr B)\times\mathscr X,\precsim)$ is
the relation obtained by the partial order on subsets of $\mathscr B$
by inclusion; its equivalence classes (strongly connected components)
have the cardinality of the continuum. We also remark that $(\mathcal
P(\mathscr B),\subseteq)$ is isomorphic to $(\mathcal P(\mathscr
B),\supseteq)$, via the map $X\mapsto\mathscr B\setminus X$.

\begin{corollary}\label{cor:order}
  A preorder imbeds in $(\markedgroups/{\cong},\conv)$ if and only if
  it imbeds in $(\mathcal P(\mathscr B)\times\mathscr X,\precsim)$. In
  particular, a partial order imbeds in
  $(\markedgroups/{\cong},\conv)$ if and only if it is realizable by
  subsets of a countable set under inclusion.
\end{corollary}
\begin{proof}
  Proposition~\ref{prop:order} yields imbeddings between
  $\markedgroups$ and $\mathcal P(\mathscr B)\times\mathscr X$. We
  therefore have an imbedding of $\markedgroups/{\cong}$ into
  $\mathcal P(\mathscr B)\times\mathscr X$.

  Conversely, isomorphism classes of groups in $\markedgroups$ are
  countable, because there are countably many homomorphisms between
  finitely generated groups. On the other hand, equivalence classes in
  $\mathcal P(\mathscr B)\times\mathscr X$ are uncountable; so there
  exists an imbedding $\mathcal P(\mathscr B)\times\mathscr
  K\to\mathcal P(\mathscr B)\times\mathscr X$, which is the identity
  on its first argument, and such that its image imbeds in
  $\markedgroups/{\cong}$.
\end{proof}

%
%
%

\begin{proof}[Proof of Proposition~\ref{prop:order}, $\hookrightarrow$]
  Consider first the space $\markedgroups$ of marked groups.  For
  every $k,R\in\N$, there are finitely many possibilities for the
  marked graphs $B(1,R)$ of degree $\le k$ that may appear in the
  Cayley graphs of these groups; letting $k,R$ range over $\N$, we
  obtain a countable collection $\mathscr B$ of finite graphs. Now to
  each $(G,S')\in\markedgroups$ we associate the subset $\mathscr O_G$
  of $\mathscr B$ consisting of all marked balls that may appear in
  Cayley graphs $\cay GS$, as we let $S$ range over generating sets of
  $G$. Clearly, $G\conv H$ if and only if $\mathscr
  O_H\subseteq\mathscr O_G$.

  We deduce that $(\markedgroups,\conv)$ imbeds in $(\mathcal
  P(\mathscr B),\subseteq)$. We can make this map injective by taking
  $\mathscr X=\mathcal P(\free)$, and mapping $(G,S)$ to $(\mathscr
  O_G,\ker(\free\twoheadrightarrow G))$, for the natural map
  $\free\twoheadrightarrow G$ presenting $G$.
\end{proof}

To construct the imbedding in the other direction, we begin by a
general construction. Let $P$ be a group. Consider first the free
nilpotent group $N_{2,P}$ of class $2$ on a generating set indexed by
$P$.  Denote its generators by $a_p$ for $p\in P$, and for $p,q\in P$
write $c_{p,q}:=[a_p,a_q]$. We have $c_{p,p}=0$, and
$c_{p,q}=-c_{q,p}$ for all $p,q\in P$. Define then $\overline N_{2,P}$
as the quotient of $N_{2,P}$ by the relations $c_{p,q}=c_{pr,qr}$ for
all $p,q,r\in P$. Finally let $H(P)$ be the semidirect product
$P\ltimes\overline N_{2,P}$, for the action $a_p\cdot q:=a_{pq}$. The
centre of $H(P)$ is generated by the images of the $c_{p,q}$. Let
$P_+\subseteq P\setminus\{1\}$ contain precisely one element out of
each pair $\{p,p^{-1}\}$; then $\{c_{1,p}\}$ freely generates the
centre of $H(P)$. If $S$ be a generating set for $P$, then
$S\cup\{a_1\}$ generates $H(P)$.

The case $P=\Z$ is considered by Hall in~\cite{hall:soluble}*{\S3}; he
introduced this group in order to construct $2^{\aleph_0}$
non-isomorphic solvable finitely generated groups (of solvability
length $3$).
  
In this proof, we take $P=\Z^2$, and for convenience
$(\Z^2)_+=\{(m,n)\in\Z^2\colon m>0\text{ or }m=0<n\}$. We abbreviate
$H(\Z^2)$ as $H$, generated by $\{x,y,a\}$ with $\{x,y\}$ the standard
generators of $\Z^2$ and $a=a_{(0,0)}$.

A \emph{prime colouring} is a function
$\phi\colon(\Z^2)_+\to\{1\}\cup\{\text{primes}\}$; it extends to a
function still written $\phi\colon\Z^2\to\Z$ by $\phi(-z)=-\phi(z)$
and $\phi(0)=0$. Given a prime colouring $\phi$, we define the
\emph{standard quotient} $H_\phi$ of $H$ as the quotient of $H$ by all
the relations $c_{1,z}^{\phi(z)}=1$, as $z$ ranges over
$(\Z^2)_+$. Clearly,
\begin{lemma}\label{lem:order1}
  A standard central quotient $H_\phi$ contains an element of order
  $p$ if any only if there exists $z\in(\Z^2)_+$ such that
  $\phi(z)=p$.

  If $H_\phi\conv H_\psi$, then the set of primes in $\psi(\Z)$ is
  contained in the set of primes in $\phi(\Z)$.\qed
\end{lemma}

Let $I$ be a set of primes. A prime colouring $\phi$ is
\emph{$I$-universal} if its values lie in $I$ and it contains every
finite $I$-colouring, in the following sense: for every $R\in\N$ and
every function $\theta:\{-R,\dots,R\}^2\cap(\Z^2)_+\to I\cup\{1\}$,
there exists $M\in \mathbf{SL}_2(\Z)$ such that $\theta(z)=\phi(M(z))$
for all $z\in\{-R,\dots,R\}^2\cap(\Z^2)_+$.

\begin{lemma}\label{lem:manyIcolourings}
  For every set $I$ of primes of cardinality $\ge2$, there exist a
  continuum of $I$-universal colourings.
\end{lemma}
\begin{proof}
  One enumerates all finite $I$-colourings, and constructs $\phi$ step
  by step. At each step, only finitely many values of $\phi$ have been
  specified, say within the box $\{-S,\dots,S\}^2$, and we want to
  extend $\phi$ using the partial colouring
  $\theta:\{-R,\dots,R\}^2\cap(\Z^2)_+\to I\cup\{1\}$. A large enough
  $M\in\mathbf{SL}_2(\Z)$ can be found such that
  $M(\{-R,\dots,R\}^2)\cap\{-S,\dots,S\}^2=\{(0,0)\}$, for example
  $M=(\begin{smallmatrix}(S+1)(S+R+1)+1 & S+1\\ S+R+1 &
    1\end{smallmatrix})$. Extend $\phi$ by setting
  $\phi(M(z))=\theta(z)$ for all
  $z\in\{-R,\dots,R\}^2\cap(\Z^2)_+$. Once this is done for all
  $R\in\N$, set finally $\phi(z)=1$ at unspecified values in
  $(\Z^2)_+$.

  To obtain a continuum of different $I$-universal colourings, note
  that countably many matrices $M_0,M_1,\dots$ were used in the
  construction, and the only condition was that they had to be
  sufficiently far away from the identity. Fix a finite-index subgroup
  $\Gamma\subset\mathbf{SL}_2(\Z)$. Then, given a subset
  $C\subseteq\N$, one may choose the matrices $M_i$ as above, and
  additionally such that $M_i\in\Gamma\Leftrightarrow i\in C$. This
  encodes $C$ into the constructed colouring.
\end{proof}

\begin{proof}[Proof of Proposition~\ref{prop:order}, $\hookleftarrow$]
  We are ready to imbed $\mathcal P(\mathscr B)\times\mathscr X$ into
  $\markedgroups$. Without loss of generality, we may assume that
  $\mathscr B$ is the set of primes $\ge10$.

  Given $X\subseteq\mathscr B$, consider $I=\{2,3\}\cup X$. By
  Lemma~\ref{lem:manyIcolourings}, there exist continuously many
  $I$-universal prime colourings $\phi_{I,C}$, parameterized by
  $C\subseteq\N$. Let $H_{X,C}$ be the central quotient
  $H_{\phi_{X\cup\{2,3\},C}}$, and note that the $(H_{X,C},\{x,y,a\})$
  are distinct points of $\markedgroups$ for distinct $(X,C)$. We have
  therefore defined an imbedding $\mathcal P(\mathscr
  B)\times\mathcal P(\N)\to\markedgroups$.

  On the one hand, if $H_{X,C}\conv H_{Y,C'}$, then $X\supseteq Y$ by
  Lemma~\ref{lem:order1}. On the other hand, if $X\supseteq
  Y\subseteq\mathscr B$ and $C,C'\subseteq\N$, then consider the prime
  colourings $\phi,\psi$ with $H_{X,C}=H_\phi$ and $H_{Y,C'}=H_\psi$,
  and choose $T=\{x,y,a\}$ as generating set of $H_\psi$. Consider an
  arbitrary $R\in\N$. Then the restriction of $\psi$ to
  $\{-R,\dots,R\}^2$ is a finite $(\{2,3\}\cup Y)$-colouring, and
  therefore a finite $(\{2,3\}\cup X)$-colouring; so there
  exists $M=(\begin{smallmatrix}a&b\\
    c&d\end{smallmatrix})\in\mathbf{SL}_2(\Z)$ such that $\psi$ and
  $\phi\circ M$ agree on $\{-R,\dots,R\}^2$. Consider the generating
  set $S=\{x^ay^b,x^cy^d,a\}$ of $H_\phi$; then the Cayley graphs
  $\cay{H_\psi}T$ and $\cay{H_\phi}S$ agree on a ball of radius $R$.
\end{proof}

\begin{remark}\label{complicatedtorsion}
  By Lemma~\ref{lem:identity}(3), if $A\conv B$ and $F$ is a finite
  subgroup of $B$, then $F$ imbeds in $A$.  In general, if $F$ is a
  torsion subgroup of $B$, this need not be true.  There exist
  finitely generated solvable groups $A\conv B$, such that $B$
  contains the divisible group $\Q/\Z$, while $A$ does not contain any
  divisible elements.
\end{remark}
\begin{proof}
  We modify the proof of Proposition~\ref{prop:order}. Before, we
  enumerated finite $I$-colourings
  $\theta:\{-R,\dots,R\}^2\cap(\Z^2)_+\to I\cup\{1\}$ and imposed the
  relations $c_{1,M(z)}^{\theta(z)}=1$, for appropriate
  $M\in\mathbf{SL}_2(\Z)$. Now, we enumerate $(\Z^2)_+$ as
  $\{p_1,p_2,\dots\}$, and we impose relations on $H$
  step-by-step. At each step, only finitely many of the $c_{1,z}$ will
  have been affected by the relations; we call the corresponding
  $z\in\Z^2$ \emph{bound}.

  For each $N=1,2,\dots$, we find $M\in\mathbf{SL}_2(\Z)$ such that
  $M(\{p_1,\dots,p_N\})$ is disjoint from all bound $z\in\Z^2$. We
  impose the relations $c_{1,M(p_1)}=1$ and
  $c_{1,M(p_i)}^i=c_{1,M(p_{i-1})}$ for all $i=2,\dots,N$. Finally, we
  set $c_{1,z}=1$ for all unbound $z\in\Z^2$.

  We call the resulting central quotient $G$, and note that it is
  solvable, and that its torsion is the subgroup generated by the
  $c_{1,z}$; this group is a direct sum of cyclic groups, and in
  particular is not divisible.

  On the other hand, let $(H,T)$ be the limit of $(G,S_M)$ in the
  space $\markedgroups$ of marked groups, along the generating sets
  $S_M=\{x^ay^b,x^cy^d,a\}$ corresponding to the matrices
  $M=(\begin{smallmatrix}a&b\\
    c&d\end{smallmatrix})\in\mathbf{SL}_2(\Z)$ used in the
  construction of $G$. Then $H$ contains a copy of $\Q/\Z$, with the
  limit of $c_{1,M(p_i)}$ playing the role of $1/i!$.
\end{proof}

\section{The connected component of free groups}

We concentrate, in this section, on those groups that either
\preform\ or are \preformed\ by free groups. Both of these classes have
already been thoroughly investigated; the first are known as ``limit
groups'', and the second as ``groups without almost-identities''.

\subsection{Limit groups}\label{ss:limitgroups}
Groups that are \preformed\ by free groups are known as ``limit
groups''. This section reviews some known facts about them; we refer
to the recent
expositions~\cites{bestvina-feighn:sela,paulin:sela,kharlampovich-myasnikov:equations}.

Benjamin Baumslag considered residually free groups
in~\cite{baumslag:residuallyfree}. An \emph{$\omega$-residually free}
groups is a group $G$ such that, for all $n$ and all distinct
$g_1,\dots,g_n\in G$, there exists a homomorphism $\pi\colon
G\twoheadrightarrow\free$ to a free group such that all
$\pi(g_1),\dots,\pi(g_n)$ are distinct. Baumslag proved in particular
that $G$ is $\omega$-residually free if and only if it is both
residually free and commutative-transitive (see
Equation~\ref{eq:commtrans}).

Remeslennikov proved in~\cite{remeslennikov:existentiallyfree} that
the following are equivalent for a residually free group: it is
$\omega$-residually free; it is universally free (namely has the same
universal theory as a free group); it is commutative transitive (see
Equation~\ref{eq:commtrans}). All three statements are
characterizations of non-abelian \emph{limit groups}. The terminology was
introduced by Sela, referring to limits of epimorphisms onto free
groups.



Champetier and Guirardel
show in~\cite{champetier-guirardel:limitgroups} that $G$ is a limit
group if and only if it is a limit of subgroups of free groups. In
other words, $G$ is a non-abelian limit group if and only if $\free_2
\conv G$.

Kharlampovich-Myasnikov~\cite{kharlampovich-myasnikov:iavofg1,kharlampovich-myasnikov:iavofg2}
and Sela~\cite{sela:diophantine1} prove that limit groups are finitely
presented.

\subsection{Groups groups with no almost-identities}\label{ss:smallerfree}
Groups that \preform\ free groups will be shown to be
``groups with no almost-identities''. We write $G\conv\free$ if there
exists $k\in\N$ such that $G\conv\free_k$; equivalently,
$G\conv\free_k$ for all $k$ large enough.

We begin by some elementary observations and examples. We include the
proofs for convenience of the reader.
\begin{lemma}[See~\cite{schleimer:girth} and~\cite{champetier-guirardel:limitgroups}*{Example~2.4(d)}]\label{lem:freemn}
  We have $\free_m\conv\free_n$ if and only if $m\le n$.

  More precisely, let $\{x_1,\dots,x_m\}$ be a basis of $\free_m$ and
  let $S_R$ be, for all $R\in\N$, a set of $n-m$ words of length at
  least $2R$ satisfying the $C'(1/6)$ small cancellation
  condition. Then $(\free_m,\{x_1,\dots,x_m\}\cup S_R)$ converges to
  $(\free_n,\text{basis})$ in $\markedgroups$.
\end{lemma}
\begin{proof}
  Consider $m\le n$. Let $S=\{x_1,\dots,x_m\}$ be a basis of
  $\free_m$.  Given $R>0$, consider a set $S_R:=\{w_1,\dots,w_{n-m}\}$
  such that each word $w_i$ has length larger than $2R$, and
  $\{w_1,\dots,w_{n-m}\}$ satisfies the $C'(1/6)$ small cancellation
  condition. The presentation $\langle
  x_1,\dots,x_m,y_1,\dots,y_{n-m}\mid
  y_1w_1,\dots,y_{n-m}w_{n-m}\rangle$ then defines the free group
  $\free_m$, and also satisfies the $C'(1/6)$ small cancellation
  condition. By Greendlinger's Lemma~\cite{greendlinger:dehn}, the
  shortest relation in it has length larger than $2R$, so the ball of
  radius $R$ in $\cay{\free_m}{\{x_1,\dots,x_m\}\cup S_R}$ coincides
  with that in $\free_n$.

  Conversely, if $\free_m\conv\free_n$ then $\free_m$ is a quotient of
  $\free_n$, by Lemma~\ref{lem:fp}, so $m\le n$.
\end{proof}

\begin{lemma}[See~\cite{schleimer:girth}*{Lemma~5.1}]\label{lem:quotientfreeprecedesfree}
  If $G$ be an $s$-generated group which admits $\free_m$ as a
  quotient, for some $m\ge2$, then $G$ \preforms\ a free
  group on $m+s$ elements.
\end{lemma}
\begin{proof}
  Let $\{g_1,\dots,g_s\}$ generate $G$, and let $g'_1,\dots,g'_s$ be
  the projections of the $g_i$ to $\free_m$.  Let also
  $h_1,\dots,h_m\in G$ project to a basis $x_1,\dots,x_m$ of
  $\free_m$. Let $N$ be the maximal length of a $g'_i$ in the basis
  $\{x_1,\dots,x_m\}$.

  For each $R>0$, consider words $w_1,\dots,w_s$ in
  $\{x_1,\dots,x_m\}$ of length at least $R$ and satisfying the small
  cancellation condition $C'(1/6)$. Consider the generating set
  $S_R=\{h_1,\dots,h_m,g_1w_1(h_1,\dots,h_m),\dots,g_sw_s(h_1,\dots,h_m)\}$
  of $G$, of cardinality $m+s$. Its projection to $\free_m$ is
  $\{x_1,\dots,x_m,g'_1w_1,\dots,g'_sw_s\}$. These elements may be
  rewritten as words of length at most $N+1$ in
  $\{x_1,\dots,x_m,w_1,\dots,w_s\}$. Therefore, by
  Lemma~\ref{lem:freemn}, no relation of length $\le R/(N+1)$ holds
  among these elements.
\end{proof}

\begin{example}\label{ex:continuum}
  For every group $A$ and every $m\ge 2$, we have
  $A\times\free_m\conv\free$, $A*\free_m\conv\free$ and
  $A\wr\free_m\conv\free$.

  In particular, there exists a continuum of non-isomorphic groups
  that \preform\ free groups.
\end{example}

\begin{remark}
  If $A$ \preforms\ a non-abelian free group, and $A$ is a
  quotient of $B$, then $B$ \preforms\ a non-abelian free
  group.
\end{remark}
\begin{proof}
  By Lemma~\ref{lem:quotients} we know that $B$ \preforms\ some
  group $C$, that admits a non-abelian free group as a quotient.  By
  Lemma~\ref{lem:quotientfreeprecedesfree} we know that $C$
  \preforms\ a non-abelian free group. Therefore, $B$
  \preforms\ a non-abelian free group.
\end{proof}

By Lemma~\ref{lem:identity}(1), if $G$ satisfies an identity then $G$
doesn't \preform\ a free group. However, this does not characterize
groups that \preform\ free groups.

\begin{lemma}
  Given words $w_1,\dots,w_\ell\in\free_n$, there exists a word
  $w\in\free_n$ such that, for every group $G$, the identity $w$ is
  satisfied in $G$ as soon as at least one identity $w_i$ is
  satisfied.
\end{lemma}
\begin{proof}
  Construct words $v_1,\dots,v_\ell\in\free_n$ inductively as follows:
  $v_1=w_1$; and for $i\ge2$, if $v_{i-1}$ and $w_i$ have a common
  power $v_{i-1}^a=w_i^b=z$ then $v_i:=z$, while otherwise
  $v_i=[v_{i-1},w_i]$.
  
  Observe that $v_\ell$ is non-trivial,
  and $v_\ell(g,h)=1$ if $w_i(g,h)=1$ for some
  $i\in\{1,\dots,\ell\}$. Therefore $w=v_\ell$ is the required
  identity.
\end{proof}

\begin{corollary}\label{lemma:withoutidentity}
  A group satisfies no identity if and only if it \preforms\ a
  group containing a non-abelian free subgroup.
\end{corollary}
\begin{proof}
  If a group $G$ satisfies an identity, than so does any group that
  is \preformed\ by it; so no group which is \preformed\ by $G$ may have a
  non-abelian free subgroup.

  Conversely, consider a group $G$ which satisfies no identity.  Let
  the set $S$ generate $G$. For every $R>0$, apply the previous lemma
  to the set $\{w_1,\dots,w_\ell\}$ of non-trivial words of length at
  most $R$ in $\free_2$. Let $w$ be the resulting identity. Since it
  does not hold in $G$, there are $g_R,h_R$ be such that
  $w(g_R,h_R)\neq1$, so $v(g_R,h_R)\neq1$ for every word $v$ of length
  at most $R$. Consider the generating set $S_R=S\cup\{g_R,h_R\}$ of
  $G$. Take a converging subsequence, in $\markedgroups$, of the
  marked groups $\cay G{S_R}$, and let $\cay HT$ be its limit. Then
  the last two elements of $T$ generate a free subgroup $\free_2$ of
  $H$.
\end{proof}

Akhmedov and Olshansky-Sapir~\cites{olshanskii-sapir:fklike,
  akhmedov:schleimer} make the following definition. Let $G$ be a
$k$-generated group. A non-trivial word $w(x_1,\dots,x_k)$ is a
\emph{$k$-almost-identity} for $G$ if $w(g_1,\dots,g_k)=1$ for all
$g_1,\dots,g_k\in G$ such that $\{g_1,\dots,g_k\}$ generates $G$. The
group $G$ \emph{satisfies an almost-identity} if for all $k\in\N$
there exists a $k$-almost-identity satisfied by $G$.

\begin{corollary}[Olshansky-Sapir, \cite{olshanskii-sapir:fklike}*{Theorem~9}]\label{cor:noai=<free}
  A group \preforms\ a free group if and only if it satisfies no
  almost-identity. More precisely, $G\conv\free_k$ if and only if $G$
  is $k$-generated and satisfies no $k$-almost-identity.
\end{corollary}
\begin{proof}
  If $G$ satisfies a $k$-almost-identity and $G\conv H$, then $H$
  satisfies the same almost-identity; therefore $H$ cannot be free.

  Conversely, consider a $k$-generated group $G$ which satisfies no
  $k$-almost-identity.  For every $R>0$, apply the previous lemma to the
  set $\{w_1,\dots,w_\ell\}$ of non-trivial words of length at most
  $R$ in $\free_k$. Let $w$ be the resulting word. Since it is not a
  almost-identity satisfied by $G$, there exists a generating set
  $S_R:=\{g_{R,1},\dots,g_{R,k}\}$ of $G$ such that
  $w(g_{R,1},\dots,g_{R,k})\neq1$, so $v(g_{R,1},\dots,g_{R,k})\neq1$
  for every word $v$ of length at most $R$. Take a converging
  subsequence, in $\markedgroups$, of the marked groups $\cay G{S_R}$,
  and let $\cay HT$ be its limit. Then $H$ is a free group of rank
  $k$.
\end{proof}

Following an idea sketched by Schleimer
in~\cite{schleimer:girth}*{\S4}, Olshansky and Sapir show
in~\cite{olshanskii-sapir:fklike} that there are groups with
almost-identities but without identities; see
also~\cite{akhmedov:schleimer}*{\S4}.
\begin{example}[Schleimer, Olshansky \& Sapir]
  There exist groups without identities, but with
  almost-identities. For all $n$ large enough, such an example is the
  group $\free_2/\langle w^n\colon
  w\not\in\free_2^n[\free_2,\free_2]\rangle$.
\end{example}

It is known that the following groups \preform\ $\free$:
\begin{enumerate}
\item Non-elementary hyperbolic groups~(see
  Akhmedov~\cite{akhmedov:girth}, with a refinement in by Olshansky
  and Sapir~\cite{olshanskii-sapir:fklike} on the number of generators
  of the free group);
\item linear groups~\cite{akhmedov:girth};
\item one-relator groups~\cite{akhmedov:girth};
\item Thompson's group $F$ (Brin shows in~\cite{brin:thompsonfreelike}
  that it \preforms\ $\free_2$, and Akhmedov, Stein and Taback
  give a slightly worse
  estimate~\cite{akhmedov-stein-taback:thompson}).
\end{enumerate}

Akhmedov also shows that there exist amenable groups that
\preform\ $\free$. We show later in this section that there are
groups of intermediate growth (e.g.\ the first Grigorchuk group) that
\preform\ free groups.

\begin{remark}\label{rem:order}
  Any order satisfying the assumption of Corollary~\ref{cor:order} is
  imbeddable in the set of groups that \preform\ $\free$.
\end{remark}
\begin{proof}
  If $G$ \preforms\ $H$, then $G\times\free_m$ \preforms\
  $H\times\free_m$, by Lemma~\ref{lem:products}(1).

  Observe, by considering the torsion subgroups, that the converse is
  true for the groups used in the proof of Proposition~\ref{prop:order}.
\end{proof}

\subsection{A criterion \`a la Ab\'ert for having no almost-identity }\label{ss:abert}
We start by recalling a general result by
Ab\'ert~\cite{abert:nonfree}. Consider a group $G$ acting by
permutations on a set $X$. Say that $G$ \emph{separates} $X$ if, for
every finite $Y\subseteq X$, the fixed point set of the fixator $G_Y$
of $Y$ is equal to $Y$. Ab\'ert proves that if $G$ separates $X$ then
$G$ satisfies no identity.

In the theorem below we strengthen the assumption of Ab\'ert's theorem
in order to get a criterion for absence of almost-identities, not only
identities. Recall that the \emph{Frattini subgroup} $\Phi(G)$ of a
group $G$ is the intersection of its maximal subgroups. It is the
maximal subgroup of $G$ such that $S$ generates $G$ if and only if
$S\Phi(G)$ generates $G/\Phi(G)$. Equivalently, if $\{s_1,\dots,s_k\}$
generates $G$, then $\{s_1g_1,\dots,s_kg_k\}$ also generates $G$, for
arbitrary $g_1,\dots,g_k\in\Phi(G)$.

\begin{theorem}\label{prop:abert}
  Let $G$ separate the set $X$ on which it acts on the right, and
  assume that $\Phi(G)$ has finite index in $G$. Then $G$ satisfies no
  almost-identity.
\end{theorem}
\begin{proof}
  We follow~\cite{abert:nonfree}*{Theorem~1}. Let $k$ be large enough
  that $G$ can be $k$-generated, and let $w=w(x_1,\dots,x_k)=v_1\dots
  v_\ell$ be a non-trivial reduced word in $\free_k$. Write
  $w_n=v_1\dots v_n$ for all $n\in\{0,\dots,\ell\}$. Fix a point
  $p_0\in X$. A tuple $(g_1,\dots,g_k)\in G^k$ is called
  \emph{distinctive} for $w$ if all the points
  $p_n=p_0w_n(g_1,\dots,g_k)$, for $n=0,\dots,\ell$, are
  distinct. This implies in particular $p_\ell\neq p_0$, so
  $w(g_1,\dots,g_k)\neq1$.

  We prove by induction on $n=0,\dots,\ell$ that there exists a
  distinctive tuple $(g_1,\dots,g_k)$ for $w_n$ such that
  $\{g_1,\dots,g_k\}$ generates $G$. The case $n=0$ follows from the
  fact that $G$ can be $k$-generated; we choose any generating
  sequence $(g_1,\dots,g_k)$.

  By induction, we may assume that $p_0,\dots,p_{n-1}$ are all
  distinct. Put
  \[Y=\{p_i\colon v_{i+1}=v_n\text{ for }i\le n-1,\text{ or }v_i=v_n^{-1}\}.
  \]
  If $v_n=x_j$, then we modify $g_j$ into $h_j:=cg_j$ for some
  $c\in\Phi(G)\cap G_Y$ to be chosen later, while if $v_n=x_j^{-1}$
  then we modify $g_j$ into $h_j:=g_jc$. In all cases, we leave the
  other $g_i$ unchanged, and write $h_i:=g_i$ for all $i\neq
  j$. Clearly $(h_1,\dots,h_k)$ still generates $G$.

  For $i=1,\dots,n-1$, we have
  $p_i=p_{i-1}v_i(g_1,\dots,g_k)=p_{i-1}v_i(h_1,\dots,h_k)$ since
  $c\in G_Y$. From $v_{n-1}\neq v_n^{-1}$ we get $p_{n-1}\not\in Y$,
  so the $G_Y$-orbit of $p_{n-1}$ is infinite, and its
  $(G_Y\cap\Phi(G))$-orbit is infinite too. Therefore, we may choose
  $c\in G_y\cap\Phi(G)$ such that
  \[p_{n-1}c\not\in\{p_iv_n(g_1,\dots,g_k)\colon i=1,\dots,n-1\},\]
  from which
  $p_n=p_{n-1}v_n(h_1,\dots,h_k)\not\in\{x_0,\dots,x_{n-1}\}$ and
  $(h_1,\dots,h_k)$ is distinctive for $w_n$.
\end{proof}

\subsection{The first Grigorchuk group}
We now show that the first Grigorchuk group $\grig$ satisfies no
almost-identity, and therefore \preforms\ $\free$. We begin
by recalling its construction.

Consider the following recursively defined transformations $a,b,c,d$
of $\{0,1\}^\infty$: for $\omega\in\{0,1\}^\infty$,
\begin{xalignat*}{2}
  (0\omega)a&=1\omega & (1\omega)a&=0\omega,\\
  (0\omega)b&=0(\omega a) & (1\omega)b&=1(\omega c),\\
  (0\omega)c&=0(\omega a) & (1\omega)c&=1(\omega d),\\
  (0\omega)d&=0\omega & (1\omega)d&=1(\omega b).
\end{xalignat*}
This action is continuous and measure-preserving; it permutes the
clopens $\{v\{0,1\}^\infty\colon v\in\{0,1\}^*$, preserving the length
of $v$. We call such actions \emph{arborical}.  The first Grigorchuk
group $\grig$ is $\langle a,b,c,d\rangle$;
see~\cites{grigorchuk:burnside,aleshin:burnside} for its origins,
and~\cite{harpe:ggt}*{Chapter~VIII} for a more recent introduction.

Recall that a group $G$ acting arborically with dense orbits on a
cantor set $\Sigma^\infty$ is \emph{weakly branched} if, for every
$v\in\Sigma^*$, there exists $g\in G$ which acts non-trivially in the
clopen $v\Sigma^\infty\subseteq\Sigma^\infty$ but fixes its
complement. It is known that $\grig$ is weakly branched.

\begin{lemma}\label{lem:wb=>sep}
  If $G$ is weakly branched, then it separates $\Sigma^\infty$.
\end{lemma}
\begin{proof}
  Consider a finite subset $Y\subset\Sigma^\infty$, and
  $\omega\in\Sigma^\infty\setminus Y$. Let $v\in\Sigma^*$ be a prefix
  of $\omega$ that is not a prefix of any element of $Y$. Let $H$
  denote the stabilizer $v\Sigma^\infty$, and let $K\triangleleft H$
  be the set of $g\in G$ that fix $\Sigma^\infty\setminus
  v\Sigma^\infty$.

  Since $G$ has dense orbits on $\Sigma^\infty$, its subgroup $H$ has
  dense orbits on $v\Sigma^\infty$. Assume for contradiction that $K$
  fixes $\omega$; then, since $K$ is normal in $H$, it fixes $\omega
  H$ which is dense in $v\Sigma^\infty$, so $K=1$, contradicting the
  hypothesis that $G$ is weakly branched.
\end{proof}

\begin{corollary}\label{cor:grig<free}
  The first Grigorchuk group $\grig$ \preforms\
  $\free_3$. In particular, $\grig$ has infinite girth.
\end{corollary}

Note that this gives a negative answer to a question of Schleimer, who
has conjectured in~\cite{schleimer:girth}*{Conjecture 6.2} that all
groups with infinite girth have exponential growth.

\begin{proof}
  Lemma~\ref{lem:wb=>sep} shows that $\grig$ separating
  $\{0,1\}^\infty$. Pervova proved in~\cite{pervova:edsubgroups} that
  all maximal subgroups of $\grig$ have index $2$; so the Frattini
  subgroup of $\grig$ satisfies
  $\Phi(\grig)=[\grig,\grig]$. Proposition~\ref{prop:abert} then shows
  that $\grig$ satisfies no almost-identity, so $\grig\conv\free_3$.
\end{proof}

Note that Pervova proved, in~\cite{pervova:maximal}, that a large
class of groups, called ``GGS groups'', satisfy the same condition
that all of their maximal subgroups are normal, and hence contain the
derived subgroup. Since all GGS groups (except a few, well-understood
exceptions) are weakly branched, they all \preform\
$\free_2$, following the same argument as in~\ref{cor:grig<free}.

\subsection{Permutational wreath products}\label{ss:wreath}
We return to wreath products, and consider a more general situation.
Let $A$ be a group, and let $G$ be a group acting on a set $X$. Recall
that the \emph{permutational wreath product} is the group
\[A\wr_X G=\{f:X\to A\text{ finitely supported}\}\rtimes G,
\]
with the standard action at the source of $G$ on functions $X\to
A$. The \emph{standard wreath product} $A\wr G$ is then the wreath in
which $X=G$ carries the regular $G$-action.

We extend the notion of Cayley graph to sets with a group action (they
are sometimes called \emph{Schreier graphs}. If $G=\langle T\rangle$,
we denote by $\cay XU$ the graph with vertex set $X$ and an edge from
$x$ to $xt$ for all $x\in X,t\in T$.

\begin{lemma}\label{lem:any<direct}
  Let $A=\langle a_1,\dots,a_k\rangle$ be an arbitrary group, and let
  $G=\langle T\rangle$ be a group acting transitively on an infinite
  set $X$. Fix a point $x_1\in X$, and assume that, for all $R\in\N$,
  there exist $x_2,\dots,x_k\in X$, at distance $>R$ from each other
  and from $x_1$ in $\cay XT$, such that the balls of radius $R$
  around all $x_i$ are isomorphic. Let $e_1,\dots,e_k$ denote the
  orders of $a_1,\dots,a_k$ respectively. Then
  \[A\wr_X G\conv(C_{e_1}\times\cdots\times C_{e_k})\wr_X G.\]
\end{lemma}
\begin{proof}
  We adapt the argument in Example~\ref{ex:wreath}.  As generating set
  of $(C_{e_1}\times\cdots\times C_{e_k})\wr_X G$, we consider
  $\{b_1,\dots,b_k\}\sqcup T$, in which $b_i$ corresponds to the
  generator of $C_{e_i}$ supported at $x_0\in X$.

  For arbitrary $R\in\N$, choose $x_1,\dots,x_k\in X$ as in the
  Lemma's hypotheses, and consider the following generating set
  $\{s_1,\dots,s_k\}\sqcup T$ of $A\wr_X G$: the generator $s_i$
  corresponds to the generator $a_i$ of the copy of $A$ supported at
  $x_i$.

  Both $\prod C_{e_i}\wr_X G$ and $A\wr_X G$ are quotients of
  $(\bigast_i C_{e_i})*G$; for the former, the additional relations are
  $[b_i,g]$ for all $i\in\{1,\dots,k\}$ and $g\in G_{x_0}$, and
  $[b_i^g,b_j]$ for all $i,j\in\{1,\dots,k\}$ and $g\in G$.

  For the latter, the additional relations are $[s_i,g]$ for all
  $i\in\{1,\dots,k\}$ and $g\in G_{x_i}$, and $[s_i^g,s_j]$ for all
  $i,j\in\{1,\dots,k\}$ and $g\in G$ with $x_ig\neq x_j$, and
  $w(s_1^{g_1},\dots,s_k^{g_k})$ for every relation
  $w(a_1,\dots,a_k)=1$ in $A$ and every $g_1,\dots,g_k\in G$ such that
  $x_ig_i=x_jg_j$ for all $i,j$.

  Our conditions imply that these two sets of relations agree on a
  ball of radius $R$.
\end{proof}

Our main example is as follows. Let $X$ be the orbit of $0^\infty$
under $\grig$.

\begin{corollary}\label{cor:grigwr}
  For every group $G$, there exists an abelian group $B$ such that
  $G\wr_X\grig\conv B\wr_X\grig$.
\end{corollary}
\begin{proof}
  Let $\{a_1,\dots,a_k\}$, of respective orders $e_1,\dots,e_k$,
  generate $G$. Define $B=C_{e_1}\times\cdots\times C_{e_k}$. Choose
  $x_1=0^\infty$, and for $R\in\N$ choose distinct words
  $v_2,\dots,v_k\in\{0,1\}^*$ of length $2\lfloor\log_2R\rfloor$. Set
  $x_i=v_i0^\infty$ for $i=2,\dots,k$. Since the action of $\grig$ is
  contracting, the $R$-balls around the $x_i$ are isomorphic. The
  conclusion follows from Lemma~\ref{lem:any<direct}.
\end{proof}

\subsection{A necessary and sufficient condition for standard wreath products}\label{ss:morewreath}
\begin{proposition}\label{prop:wreath}
  Consider a wreath product $W=G\wr H$ with $H$ infinite. Then $G\wr
  H\conv\free$ if and only if one of the following holds:
  \begin{enumerate}
  \item $G$ does not satisfy any identity;
  \item $H$ does not satisfy any almost-identity.
  \end{enumerate}
\end{proposition}

We split the proof in a sequence of lemmas.  The following generalizes
the construction in~\cite{akhmedov:girth}*{Lemma~2.3} and the main
result of that paper:
\begin{lemma}\label{lem:generalizedakhmedov}
  Let $G$ be a $k$-generated group that satisfies no identity, and let
  $H$ be an infinite group. Then $G\wr H$ \preforms\
  $\free_{k+1}*H$, and hence \preforms\ $\free$ in view of
  Lemma~\ref{lem:quotientfreeprecedesfree}.
\end{lemma}
\begin{proof}
  Fix generating sets $S=\{g_1,\dots,g_k\}$ of $G$ and $T$ of $H$; we
  then identify $g_i$ with the function $H\to G$ supported at
  $\{1\}\subset H$ at taking value $g_i$ at $1$.

  By Lemma~\ref{lemma:withoutidentity} and Lemma~\ref{lem:products}(4)
  it is sufficient to consider the case in which $G$ contains a
  non-abelian free subgroup. Given $R>0$, we construct the following
  generating set of $G\wr H$. Let $B$ denote the ball of radius
  $(k+1)R$ in $H$. Since $G$ contains a free subgroup, it also
  contains a free subgroup $\free_B$ of rank $\#B$. Let $w$ be a
  function $G\to H$, supported at $B$, whose image is a basis of
  $\free_B$. Choose also $h \in H\setminus B$, and $h_1,\dots,h_k\in H$
  such that $\|h_i\|=Ri$ for all $i=1,\dots,k$. Consider then the
  set
  \[U=\{w,w^{h_1}g_1^h,\dots,w^{h_k}g_k^h\}\cup T.
  \]
  It is clear that $U$ generates $G\wr H$. Consider a word $u$ of
  length $\le R$ in $U^{\pm1}$. Assume that it contains no relation in
  $H$ (that would come from the $T$ letters). If $u$ is non-trivial,
  then it contains at least one term $w^{h_i}g_i^h$. Concentrating on
  what happens in $B$, we see generators of $\free_B$ that cannot
  cancel, because to do so they would have to come from a term
  $(w^{h_i}g_i^h)^{-1}$, which would imply that $u$ was not reduced,
  or from a term $(w^{h_j}g_j^h)^{-1}$ via conjugation by a word of
  length at least $R$ in $T$.

  Therefore, the relations of length $\le R$ that appear in $\cay{G\wr
    H}U$ are precisely those of $\cay HT$.
\end{proof}

\begin{lemma}\label{Hwithoutquasi}
  If $H$ satisfies no almost-identity, then $G\wr H$ \preforms\ a
  non-abelian free group.
\end{lemma}
\begin{proof}
  Let $H$ be $k$-generated. Since $H$ does not satisfy any
  $k$-almost-identity, it \preforms\ $\free_k$ by
  Corollary~\ref{cor:noai=<free}. By Lemma~\ref{lem:products}(4), we
  get $G\wr H\conv G\wr\free_k$. Then $G\wr\free_k$ admits $\free_k$
  as a quotient, hence by Lemma~\ref{lem:quotientfreeprecedesfree}
  \preforms\ a non-abelian free group.
\end{proof}

If two groups satisfy an identity, then so does their wreath product.
An analogous statement is valid for almost-identities:
\begin{lemma}\label{GidentityHquasi}
  Suppose that the group $G$ satisfies an identity, and that for all
  $k\in\N$ there is an $k$-almost-identity in $H$.  Then for all
  $k\in\N$ the wreath product $G\wr H$ satisfies a
  $k$-almost-identity.
\end{lemma}
\begin{proof}
  Let $k\in\N$ be given, let $v(x_1,\dots,x_m)$ be an identity for
  $G$, and let $w(x_1,\dots,x_k)$ be an almost-identity for $H$ on
  generating sets of cardinality $k$.

  Let $\{s_1,\dots,s_k\}$ be a generating set for $G\wr H$. Its
  projection to $H$ then is a generating set for $H$, so
  $w(s_1,\dots,s_k)$ belongs to the base $G^H$ of $G\wr H$. For
  $a_1,\dots,a_m\in\free_k$ to be determined later, let us consider
  the word
  \[u(x_1,\dots,x_k)=v(w(x_1,\dots,x_k)^{a_1},\dots,w(x_1,\dots,x_k)^{a_m}).
  \]
  We clearly have $u(s_1,\dots,s_k)=1$, so $u$ is an almost-identity
  in $G\wr H$. We only have to choose the $a_i\in\Z$ in such a way
  that $u$ is not the trivial word.

  Since $w$ is a non-trivial word, there exists $a\in\free_k$ such
  that $\langle w,a\rangle$ is a free group of rank $2$. Observe that
  $\{w^{a^n}\colon n\in\N\}$ freely generates a free subgroup $E$ of
  $\free_k$. Select then $a_i=a^i$. Then, since $v$ is a non-trivial
  word, $v(w^{a_1},\dots,w^{a_m})$ is a non-trivial element of $E$ and
  therefore of $\free_k$.
\end{proof}

\begin{example}[A solvable group in the component of free groups]\label{ex:solvablefree}
  Consider $A=\free_2\wr\Z$ and $B=\Z^2\wr\Z$. Then $B$ is solvable of
  class $2$.  By Lemma~\ref{lem:any<direct}, the group $A$
  \preforms\ $B$.  Since $\free_2$ satisfies no identity and since
  $\Z$ is infinite, Lemma~\ref{lem:generalizedakhmedov} implies that
  $A$ \preforms\ a free group.

  In summary, $A$ \preforms\ a solvable group, and also
  \preforms\ a non-abelian free group.
\end{example}

\begin{example}[A group of bounded torsion in the component of free groups]\label{ex:boundedtorsionfree}
  Let $p\ge 3$ be such that there exist infinite finitely generated
  groups of $p$-exponent (any sufficiently large prime $p$ has such
  property, see~\cite{adyan:burnside-orig}). Let $H$ be an infinite
  $s$-generated group of exponent $p$. Set $A=(\bigast^s\Z/p\Z)\wr H$
  and $B=(\Z/p\Z)^s\wr H$. By Lemma~\ref{lem:any<direct}, the group
  $A$ \preforms\ $B$.

  Observe that $\bigast^s\Z/p\Z$ contains a non-abelian free subgroup
  and therefore satisfies no identity.  Since $H$ is infinite,
  Lemma~\ref{lem:generalizedakhmedov} implies that $A$ \preforms\
  a free group.  Clearly $B$ is a torsion group of exponent $p^2$.
\end{example}

\subsection{Distance between finitely generated groups}
Given two finitely generated group $A$ and $B$, let us denote by
$\dist(A,B)$ the distance between $A$ and $B$ in the (oriented) graph
corresponding to the
 limit preorder. It is the minimal length
$\ell$ of a chain of groups $A=A_0,A_1,\dots,A_\ell=B$ such that
either $A_{i-1}\conv A_i$ or $A_i\conv A_{i-1}$ for all
$i=1,\dots,\ell$. We also write $\dist(A,B)=\infty$ if $A$ and $B$ are
in distinct connected components.

If $A$ is a torsion-free nilpotent group, then we have seen in
Proposition~\ref{mainpropositionnilpotent} that the diameter of the
connected component that contains $A$ is equal to two.

Examples~\ref{ex:solvablefree} and~\ref{ex:boundedtorsionfree} exhibit
solvable groups and groups of bounded exponent at distance $2$ from
some non-abelian free group.

In contrast to the nilpotent case, the diameter of the connected
component that contains non-abelian free groups is at least $3$:
\begin{remark}\label{re:atleastthree}
  If $A$ is a finitely presented group satisfying an identity (for
  example, a finitely presented solvable group), then
  $\dist(A,\free_m)\ge 3$ for all $m\ge2$. Indeed, any group that
  is \preformed\ by $A$ satisfies the same identity. Any group that
  \preforms\ $A$ is a quotient of $A$ (since $A$ is finitely
  presented) and hence also satisfies the same identity.  This implies
  that all groups that are \preformed\ by or \preform\ $A$ are at
  distance at least $2$ from non-abelian free groups.  Therefore, the
  distance from $A$ to free groups is at least $3$.
\end{remark}
 
Before we discuss in more detail some groups from
Remark~\ref{re:atleastthree}, we need the following
 
\begin{example}\label{ex:baumslagandwreath}
  Consider $p\ge 2$, and let
  \[\mathbf{BS}(1,p)=\langle a,t\mid t^{-1}at=a^p\rangle\]
  be a solvable Baumslag-Solitar group. Then $\mathbf{BS}(1,p)$
 \preforms\ $\Z\wr\Z^2$.
\end{example} 
\begin{proof}
  We write $A=\mathbf{BS}(1,p)$. Fix sequences $(m_R),(n_R)$ in $\N$
  such that $m_R,n_R$ are relatively prime, $m_R\to\infty$,
  $n_R\to\infty$ and $n_R/m_R\to\infty$. For example, $m_R=i$ and
  $n_R=i^2+1$ will do.

  Consider the generating set $\{a,x_R=t^{n_R},y_R=t^{m_R}\}$ of
  $A$. Let us prove that $(A,S_R)$ subconverges to $\Z\wr\Z^2=\langle
  a,x,y\mid [b,c],[a,a^{x^iy^j}]\forall i,j\in\Z\langle$ in
  $\markedgroups$.

  Observe that $a,x_R,y_R$ satisfy all the relations satisfied by
  $a,x,y$ in $\Z\wr\Z^2$. Therefore, $(A,S_R)$ subconverges to a
  quotient $(\Z\wr\Z^2)/N$ of $\Z\wr\Z^2$. Furthermore, $(\langle
  t\rangle,\{x_R,y_R\})$ converges to $(\Z^2,\{x,y\})$, so $N$ maps to
  the trivial subgroup of $\Z^2$ under the natural projection
  $\Z\wr\Z^2\to\Z^2$.

  Now every element of $\Z\wr\Z^2$ may uniquely be written in the form
  $w(a,x,y)=\prod_{i,j\in\Z}a^{\ell_{i,j}x^iy^j}x^py^q$, and if this
  element maps trivially to $\Z^2$ then $p=q=0$.

  Let us therefore assume by contradiction that there exists a
  non-trivial word $w(a,x,y)=\prod_{i,j\in\Z}a^{\ell_{i,j}x^iy^j}$
  with $w(a,x_R,y_R)=1$ for all sufficiently large $R$.

  The group $A$ is isomorphic to $\Z[1/p]\rtimes\Z$, with the
  generator of $\Z$ acting on $\Z[1/p]$ by multiplication by
  $p$. Since $w(a,x,y)$ maps trivially to $\Z^2$, we have
  $w(a,x_R,y_R)\in\Z[1/p]$, and in fact under this identification
  \[w(a,x_R,y_R)=\sum_{i,j\in\Z}\ell_{i,j}p^{in_R+jm_R}.\]

  Let $(i,j)\in\Z^2$ be lexicographically maximal such that
  $\ell_{i,j}\neq0$; that is, $\ell_{i',j'}=0$ if $i'>i$ or if $i'=i$
  and $j'>j$. Set $N=\sum_{i,j\in\Z}|\ell_{i,j}|$. For $R$
  sufficiently large, we have $p^{in_R+jm_R}>Np^{i'n_R+j'm_R}$
  whenever $(i',j')\in\Z^2$ is such that $\ell_{i',j'}\neq0$. For such
  $R$, we have $|w(a,x_R,y_R)|\ge
  p^{in_R+jm_R}-\sum_{(i',j')\neq(i,j)}\ell_{i',j'}p^{i'n_R+j'm_R}>0$,
  contradicting the hypothesis that $w$ is a relation in the limit of
  $(A,S_R)$.
\end{proof}
 
\begin{example}[Groups at distance $3$ from free groups]
  The distance between solvable Baumslag Solitar groups and free
  groups is equal to $3$.
\end{example}
\begin{proof}
  Consider $p\ge 2$ and $A=\mathbf{BS}(1,p)$ a solvable
  Baumslag-Solitar group. Since $A$ is finitely presented and
  solvable, Remark~\ref{re:atleastthree} implies that the distance
  from $A$ to free groups is at least $3$.

  By Example~\ref{ex:baumslagandwreath} we know that $A$
  \preforms\ $\Z\wr\Z^2$.  Since $\Z\conv\Z^2$, we know by
  Lemma~\ref{lem:products} that $\Z\wr\Z^2 \conv \Z^2\wr\Z^2$, so $A
  \conv \Z^2\wr\Z^2$.  By Lemma~{lem:any<direct}, $\free_2\wr\Z^2
  \conv \Z^2\wr\Z^2$.  By Lemma~\ref{lem:generalizedakhmedov},
  $\free_2 \wr \Z^2$ \preforms\ a free group. We therefore have a
  chain $A\conv\Z^2\wr\Z^2\vnoc\free_2\wr\Z^2\conv\free_4$, and
  $\dist(A,\free_4)\le3$.

  On the other hand, if we had $\dist(A,\free_4)=2$ then either there
  would exist $B$ with $A\conv B\vnoc\free_4$; this is impossible
  because $B$ would then be both solvable and \preformed\ by a free
  group; or there would exist $B$ with $A\vnoc B\conv\free_4$; and
  again $B$ would be both solvable and without almost-identities.
\end{proof}

\section{Groups of non-uniform exponential growth}\label{ss:growth}
Let $G$ be a group generated by a set $S$. The \emph{growth function}
of $G$ with respect to $S$,
\[\nu_{G,S}(R)=\#B(1,R)\subseteq\cay GS,\]
counts the number of group elements that may be expressed using at
most $R$ generators. This function depends on $S$, but only mildly; if
for two functions $\gamma,\delta:\N\to\N$ one defines
$\gamma\precsim\delta$ whenever there exists a constant $k\in\N_+$
such that $\gamma(R)\le\delta(kR)$, and $\gamma\sim\delta$ whenever
$\gamma\precsim\delta\precsim\gamma$, then the $\sim$-equivalence
class of $\nu_{G,S}$ is independent of $S$.

The group $G$ has \emph{polynomial growth} if $\nu_{G,S}(R)\precsim
R^d$ for some $d$; then necessarily $G$ is virtually nilpotent and
$\nu_{G,S}(R)\sim R^d$ for some $d\in\N$,
by~\cites{gromov:nilpotent,bass:nilpotent}. On the other hand, if
$\nu_{G,S}(R)\succsim b^R$ for some $b>1$, then
$\nu_{G,S}(R)\sim2^R$ and $G$ has \emph{exponential growth}; this
happens for free groups, and more generally for groups containing a
free subsemigroup.  If $G$ has neither polynomial nor exponential
growth, then it has \emph{intermediate growth}. The existence of
groups of intermediate growth, asked by Milnor~\cite{milnor:5603}, was
proven by Grigorchuk~\cite{grigorchuk:gdegree}.

Set $\lambda_{G,S}=\lim\sqrt[R]{\nu_{G,S}(R)}$; the limit exists
because $\nu_{G,S}$ is submultiplicative
($\nu_{G,S}(R_1+R_2)\le\nu_{G,S}(R_1)\nu_{G,S}(R_2)$). Reformulating
the above definitions, we say $G$ that has \emph{subexponential
  growth} if $\lambda_{G,S}=1$ for some and hence all $S$; that $G$
has \emph{exponential growth} if $\lambda_{G,S}>1$; and that $G$ has
\emph{uniform exponential growth} if $\inf_S\lambda_{G,S}>1$. The
existence of groups of non-uniform exponential growth, asked by
Gromov~\cite{gromov:metriques}*{Remarque~5.12}, was proven by
Wilson~\cite{wilson:ueg}.

\begin{lemma}\label{lem:lambdainf}
  If $G\conv H$, then $\inf_S\lambda_{G,S}\le\inf_T\lambda_{H,T}$. In
  particular, if $G$ has exponential growth and $H$ has subexponential
  growth, then $G$ has non-uniform exponential growth.
\end{lemma}
\begin{proof}
  For every $\epsilon>0$, there exists a generating set $T$ for $H$
  such that $\lambda_{H,T}<\inf_{T'}\lambda_{H,T'}+\epsilon$. There
  exists then $R\in\N$ such that
  $\nu_{H,T}(R)^{1/R}\le\lambda_{H,T}+\epsilon$. Choose then a
  generating set $S$ for $G$ such that the balls of radius $R$ in
  $\cay GS$ and $\cay HT$ agree. Then
  $\nu_{G,S}(R)=\nu_{H,T}(R)$, so
  $\lambda_{G,S}\le\nu_{H,T}(R)^{1/R}$ because growth functions are
  submultiplicative. Therefore, for all $\epsilon>0$ there exists $S$
  generating $G$ such that
  $\lambda_{G,S}\le\inf_{T'}\lambda_{H,T'}+2\epsilon$.
\end{proof}

Note that the inequality in Lemma~\ref{lem:lambdainf} can be strict;
for example, the Grigorchuk group $\grig$, has intermediate growth,
yet $\grig\conv\free_3$.

\begin{corollary}\label{cor:nueg}
  For every group $G$ of exponential growth, the group $G\wr_X\grig$
  has non-uniform exponential growth.
\end{corollary}
\begin{proof}
  From Corollary~\ref{cor:grigwr} we get $G\wr_X\grig\conv
  B\wr_X\grig$ for an abelian group $B$. It was proved
  in~\cite{bartholdi-erschler:permutational}*{Theorem~A} that
  $B\wr_X\grig$ has subexponential growth, in fact of the form
  $\exp(R^\alpha)$ if $B$ is finite, non-trivial, and of the form
  $\exp(R^\alpha\log R)$ if $B$ is infinite, for some constant
  $\alpha<1$, see Corollary~\ref{cor:imbednueg}. The claim then
  follows from Lemma~\ref{lem:lambdainf}.
\end{proof}

\begin{corollary}\label{cor:imbednueg}
  Every countable group may be imbedded in a group of non-uniform
  exponential growth.

  Furthermore, let $\alpha\approx0.7674$ be the positive root of
  $2^{3-3/\alpha} + 2^{2-2/\alpha} + 2^{1-1/\alpha} = 2$.  Then the
  group of non-uniform exponential growth $G$ has the following
  property: there is a constant $K$ such that, for any $R>0$, there
  exists a generating set $S$ of $G$ with
  \[\nu_{G,S}(r) \le \exp(K r^\alpha)\text{ for all }r\le R.
  \]

  In particular, there exist groups of non-uniform exponential growth
  that do not imbed uniformly into Hilbert space.
\end{corollary}
\begin{proof}
  Let $G$ be a countable group. Imbed first $G$ into a finitely
  generated group $H$. Without loss of generality, assume that $H$ has
  exponential growth (if needed, replace $H$ by $H\times\free_2$), and
  that the generators of $H$ are torsion elements.

  By Corollary~\ref{cor:grigwr}, the group $H\wr_X\grig$
  \preforms\ $A\wr_X\grig$ for a finite abelian group $A$. Since
  $A\wr_X\grig$ has growth $\sim\exp(R^\alpha)$, the first claim
  follows.

  The second claim follows from the first, since there exist groups
  $G$ that do not imbed into Hilbert space~\cite{gromov:rwrg}; and the
  property of not imbedding into Hilbert space is inherited from
  subgroups.
\end{proof}

Brieussel asked in~\cite{brieussel:entropy}*{after Proposition 2.5}
whether there exist groups of non-uniform exponential growth and
without the Haagerup property.  Recall that a group has the Haagerup
property if it admits a proper affine action on Hilbert space; this
property is also known as ``a-T-menability'',
see~\cite{cherix+-haagerup}. It is clear that any group with the
Haagerup property can be uniformly imbedded into Hilbert space.
Therefore, Corollary~\ref{cor:imbednueg} implies in particular that
there exist groups of non-uniform exponential growth that do not have
the Haagerup property.

\subsection{Non-uniform non-amenability}
Let $G$ be a group generated by a finite set $S$. By F\o lner's
criterion, $G$ is \emph{non-amenable} if the isoperimetric constant
\[\alpha_S:=\inf_{F\subset G\text{ finite}}\#(FS\setminus F)/\#F
\]
satisfies $\alpha_S>0$. Arzhantseva et
al.~\cite{arzhantseva-b-l-r-s-v:una} call $G$ \emph{non-uniformly
  non-amenable} if $G$ is non-amenable, but $\inf_S\alpha_S=0$.

If $G$ has non-uniform exponential growth and is non-amenable, then it
is non-uniformly amenable. However, there are groups of uniform
exponential growth that are non-uniformly non-amenable. Clearly, if
$G$ \preforms\ an amenable group, then $G$ may not be
uniformly non-amenable:
\begin{example}
  $\free_2\wr\Z$ has uniform exponential growth, but is
  non-uniformly non-amenable.
\end{example}
\begin{proof}
  The group $\free_2\wr\Z$ maps onto $\Z^2\wr\Z$, which is solvable
  and of exponential growth; so its growth is uniformly exponential,
  and the same holds for $\free_2\wr\Z$.

  By Lemma~\ref{lem:any<direct}, we also have
  $\free_2\wr\Z\conv\Z^2\wr\Z$, so $\free_2\wr\Z$ precedes an amenable
  group, so is not uniformly non-amenable.
\end{proof}

\section{Open problems and questions}

\begin{question}
  Is every non-virtually nilpotent group in the connected component of
  the free group?
\end{question}

A positive answer to the following question would imply a negative
answer to the question by Olshansky: ``Is there a variety
other than virtually nilpotent or free in which the relatively free
group is finitely presented?''

\begin{question}
  Do two nilpotent groups belong to the same connected
  component if and only if they have the same positive universal
  theory?
\end{question}

We have answered positively the question above in the case of
nilpotent groups $G$ such that $G$ and $G/\Torsion(G)$ generate the
same variety.

%

We show in Remark~\ref{re:atleastthree} that the diameter of the free
group's component is at least three:
\begin{question}
  What is the diameter of the connected component of the free group?
\end{question}

The following question complements the previous one; we show in
Proposition~\ref{mainpropositionnilpotent} that its answer is
positive, in particular, in the case of torsion-free nilpotent groups.
Guyot considered limits of dihedral groups in~\cite{guyot:dihedral},
and showed that they are semidirect products of (a finitely generated
abelian group with cyclic torsion subgroup) by $\Z/2$, the latter
acting by $-1$. His result implies that the groups \preformed\ by the
infinite dihedral group form a directed set.
\begin{question}
  Is every connected component of virtually nilpotent groups directed,
  namely, is it a partially ordered set in which every finite subset
  has an upper bound?
\end{question}

If $G\conv\free_k$, then there are generating sets $S_n$ for $G$, of
cardinality $k$, such that the girth of $\cay G{S_n}$ tends to
infinity.  
\begin{question}\label{qu:girth}
  If a finitely generated group $G$ has infinite girth, does one have
  $G\conv\free_k$ for some $k\in\N$?
\end{question}
In other words, the question asks whether in the definition of girth
one can always chose a sequence of generating sets with a bounded
number of generators.

Cornulier and Mann asked
in~\cite{cornulier-mann:rflaws}*{Question~18}: ``Does there exist a
group of intermediate growth that satisfies an identity?'' The
following question is also open: ``Does there exist a group of
non-uniform exponential growth that satisfies an identity?'' So as to
better determine which groups \preform\ free groups, we ask:
\begin{question}
  Does there exist a group of intermediate growth that satifies an
  almost-identity?  Does there exist a group of non-uniform
  exponential growth that satisfies an almost-identity?
\end{question}

A well-known question by S.I. Adyan asks: ``Are there finitely
presented groups of intermediate growth?'' Such a group would not be
\preformed\ by a group of exponential growth. The following question
by A. Mann is also open~\cite{mann:growthfree}*{Problem~4}: ``Are
there finitely presented groups of non-uniform exponential growth?''

Given a group $G$ of non-uniform exponential growth, it admits
generating sets $S_n$ with growth rate tending to $1$. If furthermore
the cardinalities of the $S_n$ are bounded, then a subsequence of
$(G,S_n)$ converges to a group of intermediate growth.
\begin{question}\label{qu:nue=>seg}
  Does there exist a group of non-uniform exponential growth that
  doesn't \preform\ a group of subexponential (equivalently,
  intermediate) growth?
\end{question}

\begin{question}
  Does there exist a group $G$ such that, for every finitely generated
  group $A$ of non-polynomial growth, there exists a group $H$ with
  $G\conv H$ and the growth of $A$ and $H$ are equivalent?
\end{question}


\end{document}